\documentclass[11pt, leqno]{amsart}
\usepackage{amssymb}
\usepackage{amsmath,amscd}
\usepackage[initials,nobysame,alphabetic]{amsrefs}
\usepackage{amsmath, amssymb}
\usepackage{amsfonts}
\usepackage{mathrsfs}
\usepackage{mathpazo}
\usepackage[arrow,matrix,curve,cmtip,ps]{xy}

\usepackage{amsthm}

\usepackage{float}
\usepackage{hyperref}
\usepackage{graphics}
\usepackage{graphicx}
\usepackage{verbatim}
\usepackage{multirow}
\usepackage{tikz}
\usepackage{enumerate}
\usepackage{eufrak}
\allowdisplaybreaks

\newtheorem{theorem}{Theorem}[section]
\newtheorem{lemma}[theorem]{Lemma}
\newtheorem{proposition}[theorem]{Proposition}
\newtheorem{corollary}[theorem]{Corollary}

\newtheorem*{theorem*}{Theorem}
\newtheorem{remark}[theorem]{Remark}

\newtheorem{definition}[theorem]{Definition}
\newtheorem{example}[theorem]{Example}

\def\as{\hbox{\rm a.s.{ }}}

\numberwithin{equation}{section}

\newcommand{\E}{\mathbb{E}}
\newcommand{\R}{\mathbb{R}}
\newcommand{\Ff}{\mathbb{F}}
\newcommand{\U}{\mathcal{U}}
\newcommand{\Ll}{\mathcal{L}}
\newcommand{\V}{\mathcal{V}}

\newcommand{\D}{\mathcal{D}}
\newcommand{\F}{\mathcal{F}}

\newcommand{\Om}{\Omega}

\newcommand{\Prob}{\mathbb{\Prob}}

\newcommand{\mytilde}{\raise.17ex\hbox{$\scriptstyle\mathtt{\sim}$}}

\def\a{\alpha}             %\def\P{\varPhi}       % \def\cG{\Cal G}
\def\b{\beta}                                    % \def\C{\Cal C}
\def\g{\gamma}                                   % \def\D{\Cal D}
\def\ep{\varepsilon}       %\def\G{\Gamma}        % \def\M{\Cal M}
\def\s{\sigma}                    %\def\cL{\Cal L}
\def\t{\theta}                    %\def\cO{\Cal O}

\def\l{\lambda}

\def\e{\eta}

\def\wt{\widetilde}

\def\ms{\medskip} 
\def\no{\noindent}

\begin{document}
\title[Optimal control and zero-sum games for mean-field Markov chains]{Optimal control and zero-sum games for Markov chains of mean-field type}

\author{Salah Eddine Choutri, Boualem Djehiche and Hamidou Tembine}

\address{Department of Mathematics \\ KTH Royal Institute of Technology \\ 100 44, Stockholm \\ Sweden}
\email{boualem@kth.se}
\email{choutri@kth.se}
\address{New York University, 19 Washington Square North New York, NY 10011, USA} \email{tembine@nyu.edu}
\thanks{{\bf Acknowledgements}. We would like to thank Said Hamad{\`e}ne for his insightful remarks that helped improve the content of the paper.}

\date{This version August 25, 2017}

\subjclass[2010]{60H10, 60H07, 49N90}

\keywords{Mean-field, nonlinear Markov chain, Backward SDEs, optimal control, Zero-sum game,  Saddle point, Stochastic maximum principle, Thinning}

\begin{abstract} We establish existence of Markov chains of mean-field type with unbounded jump intensities by means of a fixed point argument using the Total Variation distance. We further show existence of nearly-optimal controls and, using a Markov chain backward SDE approach, we suggest conditions for existence of an optimal control and a saddle-point for respectively a control problem and a zero-sum differential game associated with payoff functionals of mean-field type, under dynamics driven by such Markov chains of mean-field type.
\\

\end{abstract}

\maketitle

\tableofcontents

%%%%%%%%%%%%%%%%%%%%%%%%%%%%%%%%%%%%%%%%%%%%%%%%
\section{Introduction}

A Markov chain of mean-field type (also known as nonlinear Markov chain) is a pure jump process with a discrete state space whose jump intensities further depend on the marginal law of the process. It is obtained as the limit of a system of pure jump processes with mean-field interaction, when the system size tends to infinity. The marginal law of the nonlinear process, obtained as a deterministic limit of the sequence of empirical distribution functions representing the states of the finite systems, satisfies  a 'nonlinear' Fokker-Planck or masters equation called the {\it McKean-Vlasov equation}. In a sense, it represents the law of a typical trajectory in the underlying collection of interacting jump processes. In particular, optimal control and games based on the nonlinear process dynamics would give an insight into the effect of the design of control and game strategies for large system of interacting jump processes.

This class of processes is widely used for modeling purposes in chemistry, physics, biology and economics. Nicolis and Prigogine \cite{NP} were among the first authors to propose such a class of nonlinear processes as a mean-field model of a chemical reaction with spatial diffusion. It plays the same role as nonlinear diffusion processes play in the study of diffusion equations and more generally PDEs driven by nonlocal operators,  with mean-field interaction (see Sznitman \cite{Szn} and Jourdain {\it et al.} \cite{J} and the references therein). Mean-field models of the so-called first and second Schl\"ogl processes \cite{Sch} and the auto-catalytic process, which are widely used to model chemical reactions, provide interesting examples of Markov chains of mean-field type with unbounded jump intensities, and have been studied in depth in Dawson and Zheng \cite{DZ}, Feng and Zhang \cite{FZ} and Feng \cite{Fe}. These nonlinear processes are  obtained as  limits of systems of birth and death processes with mean field interaction. For application in the spread of epidemics see e.g. L{\'e}onard \cite{Leo1}, Djehiche and Kaj \cite{DK} and Djehiche and Schied \cite{DS}. For an account of existence and uniqueness of such nonlinear jump processes with bounded jump intensities we refer to Oelschl\"ager \cite{Oel}. See \cite{Leo2} for the case of unbounded jumps.

In the study of mean-field models, it is more or less decisive to make the right choice of an adequate distance (among  many others) on the set of probability measures which carries the topology of weak convergence. The total variation distance is usually the natural one to use in the study of standard Markov chains and is easy to manipulate. But, the fact that it does not necessarily guarantee finite moments (except when the state-space is finite), it may not be suitable for mean-field models when the mean-field interaction is given by e.g. the mean or the second moment, whereas the Wasserstein distance is designed to guarantee finite moments.

Nonetheless, In this paper we formulate our findings using  the  total variation distance only. We first give another proof of existence and uniqueness of Markov chains of mean-field type using a fixed point argument. The proof is based on a Girsanov-type change of measure and the Csisz{\'a}r-Kullback-Pinsker inequality. As we will see below, the full use of the total variation distance requires $L^2$-boundedness of the Girsanov density, which is insured by imposing an extra regularity condition of the intensity matrix of the Markov chain (see (A6) and (A7) below) compared with what should be natural if the Wasserstein distance is used. Furthermore, we consider optimal control and zero-sum games associated with payoff functionals of mean-field type, when the nonlinear Markov chain is controlled through its jump intensities. More precisely, we consider pure jump processes $x$ whose (eventually unbounded) jump intensities at time $t$ depend on the whole path over the time interval $[0,T]$ and also on the marginal law of $x(t)$, as long as they are predictable. In a sense, this way of constructing a nonlinear jump process is a generalization of the classical thinning procedure of a point process. 
A similar program for controlled diffusion processes is performed in \cite{DH}, with obvious overlap in the used methods and techniques.

The main results on optimal control and zero-sum games are derived using techniques involving Markov chain backward stochastic differential equations (BSDE), where existence of an optimal control and a saddle-point strategy  of the game boil down to finding a minimizer and a min-max of an underlying Hamiltonian $H$. Since the mean-field coupling through the marginal law of the controlled Markov chain makes the Hamiltonian $H$, evaluated at time $t$, depend on the whole path of the control process over the time interval $[0,t]$, we cannot follow the frequently used procedure in standard optimal control and perform a {\it deterministic} minimization of  $H$ over the set of actions $U$ and then apply a Bene{\v{s}-type progressively measurable selection theorem to produce an optimal control. We should rather take  the essential infimum of $H$ over the set $\U$ of progressively measurable controls. This nonlocal feature of the dependence of $H$ on the control does not seem covered by the existing powerful  measurable selection theorem. Therefore, our main results are  formulated {\it by assuming} existence of an essential minimum $u^*\in \U$ of $H$ and use suitable comparison results of Markov chain BSDEs to show that $u^*$ is in fact an optimal control, simply because don't know of any suitable measurable selection theorem that would guarantee  existence of an essential minimizer of $H$. One should solve this problem on case-by-case basis. Nevertheless, we give an example where we show that an optimal control exists provided  the set of Girsanov densities, indexed by admissible controls, is weakly sequentially compact. This property is satisfied if e.g. the set of intensities is mean-field free and satisfies the so-called Roxin's condition.  On the other hand, existence of a nearly-optimal control is guaranteed if we endow $\U$ with the Ekeland's distance which makes is a complete metric space and require $L^2$-boundedness of the Girsanov density.

After a section of preliminaries, we introduce in Section 3 the class of Markov chains of mean-field type and prove their existence and uniqueness under rather weak conditions on the underlying unbounded jump intensities.  In section 4, we consider the optimal control problem and provide conditions for existence of an optimal control. We further prove existence of nearly-optimal controls. Finally, in Section 5, we consider a related zero-sum game and derive conditions for existence of a saddle-point under the so-called Isaacs' condition.

\section{Preliminaries}
Let $I=\{0,1,2,\ldots\}$ equipped with its discrete topology and $\s$-field and let $\Om:=\D([0,T],I)$ be the space of functions from $[0,T]$ to $I$ that are right continuous with left limits at each $t\in [0,T)$ and are left continuous at time $T$. We endow $\Om$ with the Skorohod metric $d_0$ so that $(\Om,d_0)$ is a complete separable metric (i.e. Polish) space.  Given $t\in [0,T]$ and $\omega\in\Om$, put  $x(t,\omega)\equiv\omega(t)$ and denote by $\F^0_t:=\sigma(x(s),\,\, s\le t),\, 0\le t\le T,$ the filtration generated by $x$.  Denote by $\F$ the Borel $\sigma$-field over $\Om$. It is well known that $\F$ coincides with $\sigma( x(s),\,\, 0\le s\le T)$. Set, for $t\in [0,T]$, $|x|_t:=\underset{0\le s\le t}\sup|x(s)|$ and 
$\|a\|^2:=\underset{i,j: \, j\neq i}\sum |a_{ij}|^2$ for $a=(a_{ij},\,\, i,j\in I, \, j\neq i)$. 

To $x$ we associate the indicator process $I_i(t)=\mathbf{1}_{\{x(t)=i\}}$ whose value is $1$ if the chain is in state $i$ at time $t$ and $0$ otherwise, and the counting processes $N_{ij}(t),\,\,i\ne j$, independent of $x(0)$, such that 
$$
N_{ij}(t)=\#\{\tau\in(0,t]:x(\tau^-)=i, x(\tau)=j\},\quad N_{ij}(0)=0,
$$
which count the number of jumps from state $i$ into state $j$ during the time interval $(0,t]$. Obviously, since $x$ is right continuous with left limits, both $I_i$ and $N_{ij}$ are right continuous  with left limits. Moreover, by the relationship
\begin{equation}\label{x-rep-1}
x(t)=\sum_i iI_i(t),\quad I_i(t)=I_i(0)+\underset{j:\, j\neq i}\sum\left(N_{ji}(t)-N_{ij}(t)\right),
\end{equation}
the state process, the indicator processes, and the counting processes carry the same information which is represented by the natural filtration  $\Ff^0:=(\F^0_t,\, 0\le t\le T)$ of $x$. Note that \eqref{x-rep-1} is equivalent to the following useful representation
\begin{equation}\label{x-rep-2}
x(t)=x(0)+\sum_{i,j: \,i\neq j} (j-i) N_{ij}(t).
\end{equation}

 Below,  $C$ denotes a generic positive constants which may change from line to line.

\subsection{Markov chains} Let $G(t)=(g_{ij}(t),\,i,j\in I),\,0\le t\le T,$ be a $Q$-matrix, so that 
\begin{equation}\label{G}
g_{ij}(t)\ge 0, \,\,\, i\neq j,\quad  g_{ii}(t)=-\underset{j:\, j\neq i}\sum g_{ij}(t),\quad \underset{i,j: \, j\neq i}\sum \int_0^Tg_{ij}(t)\, dt<+\infty.
\end{equation}
In view of e.g. Theorem 4.7.3 in \cite{EK}, or Theorem 20.6 in \cite{RW} (for the finite state-space and time independent case), given the $Q$-matrix $G$ and a probability measure $\xi$ over $I$,  there exists a unique probability measure $P$ on $(\Om,\F)$ under which the coordinate process $x$ is a time-inhomogeneous Markov chain with intensity matrix $G$ and starting distribution $\xi$  i.e. such that $P\circ x^{-1}(0)=\xi$. Equivalently, $P$ solves the martingale problem for $G$ with initial probability distribution $\xi$ meaning that,   
for every $f$ on $I$, the process defined by
\begin{equation}\label{f-mart-1}
M_t^f:=f(x(t))-f(x(0))-\int_{(0,t]}(G(s)f)(x(s))\,ds
\end{equation}
is a local martingale relative to $(\Om,\F,\Ff^0)$, where 
$$
G(s)f(i):=\sum_j g_{ij}(s)f(j)=\sum_{j: \,j\neq i}g_{ij}(s)(f(j)-f(i)),\,\,\, i\in I,
$$
and
\begin{equation}\label{G-f}
G(s)f(x(s))=\sum_{i,j: \,j\neq i}I_i(s)g_{ij}(s)(f(j)-f(i)).
\end{equation}
By Lemma 21.13 in \cite{RW}, the compensated processes associated with the counting processes $N_{ij}$, defined by
 \begin{equation}\label{mart-1}
M_{ij}(t)=N_{ij}(t)-\int_{(0,t]} I_i(s^-)g_{ij}(s^-)\, ds,\quad M_{ij}(0)=0,
\end{equation}
are zero mean, square integrable and mutually orthogonal $P$-martingales  whose  predictable quadratic variations are
\begin{equation}\label{mart-2}
\langle M_{ij}\rangle_t=\int_{(0,t]} I_i(s^-)g_{ij}(s^-)\, ds.
\end{equation} 
Moreover, at jump times $t$, we have
\begin{equation}\label{mart-3}
\Delta M_{ij}(t)=\Delta N_{ij}(t)=I_i(t^-)I_j(t).
\end{equation}
Thus, the optional variation of $M$ 
$$
[M](t)=\sum_{0<s\le t}|\Delta M(s)|^2=\underset{0<s\le t}\sum\,\underset{i,j:\, j\neq i}\sum|\Delta M_{ij}(s)|^2
$$
is
\begin{equation}\label{optional-M}
[M](t)=\underset{0<s\le t}\sum\,\underset{i,j:\, j\neq i}\sum I_i(s^-)I_j(s).
\end{equation}
We call  $M:=\{M_{ij},\,\, i\neq j\}$ the accompanying martingale of the counting process $N:=\{N_{ij},\,\, i\neq j\}$ or of the Markov chain $x$.

We denote by $\Ff:=(\F_t)_{0\le t\le T}$ the completion of $\Ff^0=(\F^0_t)_{t\le T}$ with the $P$-null sets of $\Omega$. Hereafter, a process from $[0,T]\times\Om$ into a measurable space is said predictable (resp. progressively measurable) if it is predictable (resp. progressively measurable) w.r.t. the predictable $\sigma$-field on $[0,T]\times \Om$ (resp. $\Ff$).

For a real-valued matrix $m:=(m_{ij},\, i,j \in I)$ indexed by $I\times I$, we let 
\begin{equation}\label{g-t}
\|m\|_g^2(t):=\underset{i,j:\, i\neq j}\sum |m_{ij}|^2g_{ij}\mathbf{1}_{\{w(t^-)=i\}}<\infty.
\end{equation}
If $m$ is time-dependent, we simply write $\|m(t)\|_g^2$.

Let $(Z_{ij}, \, i\neq j)$ be a family of predictable processes and set 
\begin{equation}\label{norm}
\|Z(t)\|^2_{g}:=\sum_{i,j:\, i\neq j}Z^2_{ij}(t)I_{i}(t^-)g_{ij}(t^-),\quad 0< t\le T,
\end{equation}
\begin{equation}\label{quadratic}
\sum_{0<s\le t}Z(s)\Delta M(s):=\sum_{0<s\le t}\, \underset{i,j: \,i\neq j}\sum\,Z_{ij}(s)\Delta M_{ij}(s). 
\end{equation}
Consider the local martingale
\begin{equation}\label{stoch}
W(t)=\int_0^tZ(s)dM(s):=\sum_{i,j: \, i\neq j}\int_0^t Z_{ij}(s)dM_{ij}(s). 
\end{equation}
Then, the optional variation of the local martingale $W$ is
\begin{equation}\label{optional}
[W](t)=\sum_{0<s\le t}|Z(s)\Delta M(s)|^2=\sum_{0<s\le t}\,\sum_{i,j: \,i\neq j}|Z_{ij}(s)\Delta M_{ij}(s)|^2
\end{equation}
and its compensator is
\begin{equation}\label{compensator}
\langle W\rangle_t=\int_{(0,t]}\|Z(s)\|^2_{g}ds.
\end{equation}
Provided that
\begin{equation}\label{Z-g-int}
E\left[\int_{(0,T]} \|Z(s)\|^2_{g}ds\right]<\infty,
\end{equation}
$W$ is a square-integrable martingale and its optional variation satisfies 
\begin{equation}\label{isometry}
E\left[[W](t)\right]=E\left[\sum_{0<s\le t}|Z(s)\Delta M(s)|^2\right]=E\left[\int_{(0,t]} \|Z(s)\|^2_{g}ds\right].
\end{equation} 
Moreover, the following Doob's inequality holds:
\begin{equation}\label{Doob}
E\left[\sup_{0\le t\le T}\left|\int_0^tZ(s)dM(s)\right|^2\right]\le 4 E\left[\int_{(0,T]} \|Z(s)\|^2_{g}ds\right].
\end{equation}
If $\wt Z$ is another predictable process that satisfies \eqref{Z-g-int}, setting
\begin{equation}
\langle Z(t),\wt Z(t)\rangle_{g}:=\underset{i,j: \,i\neq j}\sum\,Z_{ij}(t)I_{i}(t^-)g_{ij}(t^-),\quad 0\le t\le T,
\end{equation}
and considering the martingale
$$
\wt W(t)=\int_0^t\wt Z(s)dM(s):=\sum_{i,j: \, i\neq j}\int_0^t\wt Z_{ij}(s)dM_{ij}(s),
$$
it is easy to see that
\begin{equation}\label{covariation}
E\left[ [W,\wt W](t)\right]=E\left[\int_{(0,t]} \langle Z(s),\wt Z(s)\rangle_{g}ds\right].
\end{equation} 

Since, the filtration $\mathbb{F}$ generated by the chain $x$ is the
same as the filtration generated by the family of counting processes $\{N_{ij},\, i\neq j\}$, we state the following martingale
representation theorem (see e.g. \cite{bremaud}, Theorem T11 or \cite{RW}, IV-21, Theorem 21.15).

\begin{proposition}[Martingale representation theorem]\label{mart-rep}  If $\Ll$ is a (right-continuous) square-integrable $\mathbb{F}$-martingale, there exists a  unique ($dP\times g_{ij}(s^-)I_{i}(s^-)ds$-almost everywhere) family of predictable processes $Z_{ij}, \, i\neq j,$ satisfying
\begin{equation}\label{Z}
E\left[\int_{(0,T]} \|Z(s)\|^2_{g}ds\right]<+\infty,
\end{equation}
where
\begin{equation}
\|Z(t)\|^2_{g}:=\sum_{i,j:\, i\neq j}Z^2_{ij}(t)I_{i}(t^-)g_{ij}(t^-),\quad 0\le t\le T,
\end{equation}
such that 
\begin{equation}\label{mart-rep-L}
\Ll_t=\Ll_0+\int_0^tZ(s)dM(s), \quad 0\le t\le T,
\end{equation}
where 
$$
\int_0^tZ(s)dM(s):=\sum_{i,j:\, i\neq j}\int_0^t Z_{ij}(s)dM_{ij}(s).
$$
In particular, at jump times $t$, we have
$$
\Delta \Ll_t=\sum_{i,j: \,i\neq j} Z_{ij}(t)\Delta M_{ij}(t)=\sum_{i,j: \,i\neq j} Z_{ij}(t)I_{i}(t^-)I_{j}(t).
$$ 
\end{proposition}

Next, we give an important application of Theorem \ref{mart-rep} to the local martingale $M^f$ given by \eqref{f-mart-1} where an explicit form of the process $Z$ can be displayed in terms of the function $f$. At jump times $t$, we have
$$
 \Delta M_t^f=M_t^f-M_{t^-}^f=\sum_i I_i(t^-)\sum_{j: \, j\neq i}I_j(t)(f(j)-f(i))=\sum_{i,j: \, j\neq i}(f(j)-f(i))\Delta M_{ij}(t),
$$ 
since, by \eqref{mart-3}, at a jump time $t$, $I_i(t^-)I_j(t)=\Delta M_{ij}(t)$ and $I_i(t^-)I_i(t)=I_{\{x(t^-)=i,x(t)=i\}}=0$. We may now define $Z^f(t)=(Z^f_{ij}(t))_{ij}$ by
\begin{equation}\label{deter-z-f}
Z^f_{ij}(t):=f(j)-f(i),\quad i,j\in I,
\end{equation} 
to obtain 
\begin{equation}\label{f-mart-2}
M^f_t=\int_{(0,t]}Z^f(s)dM(s)=\underset{i,j: \, j\neq i}\sum \,(f(j)-f(i))M_{ij}(t),\quad 0\le t\le T.
\end{equation}
Provided that $\underset{i,j: \, j\neq i}\sum(f(j)-f(i))^2 \int_0^Tg_{ij}(t)\, dt <+\infty$, $M^f$ is a square-integrable martingale.

For later use we need the following exponential estimate.
\begin{lemma}\label{exp-est}
Assume further that there exists $\a >0$ such that the $Q$-matrix $G=(g_{ij})$ satisfies 
\begin{equation}\label{exp-1}
\sum_{i,j: \,i\neq j} \int_0^Te^{\a |j-i|} g_{ij}(s)ds <+\infty.
\end{equation} 
If there exists a constant $\beta\ge \a$ such that 
\begin{equation}\label{exp-1-1}
E[e^{\beta x(0)}]<+\infty,
\end{equation}
then 
\begin{equation}\label{exp-2}
E[e^{\frac{\a}{2}|x|_T}]\le \kappa_0
\end{equation}
where
$$
\kappa_0:=E[e^{\a x(0)}]^{1/2}\exp{\frac{1}{2}\sum_{i,j: \,i\neq j}\int_0^T \left(e^{\a|j-i|}-1\right)g_{ij}(s)ds}.
$$
In particular, for any $q\ge 1$, there exists a positive constant $C_q$ which depends only on $q, \a$ and $\kappa_0$ such that 
\begin{equation}\label{q-norm}
E[|x|_T^q]\le C_q.
\end{equation}

\end{lemma}
\begin{proof}

We note that by \eqref{G}, using the same argument as \cite{EK}, Theorem 6.4.1, the sequence of point processes $N_{ij}, \, j\neq i$ can be represented (in distribution) as
$$
N_{ij}(t)=N^0_{ij}\left(\int_0^t g_{ij}(s)ds\right),
$$
where $N^0_{ij}, \, j\neq i,$ is a sequence of independent Poisson processes with intensity 1, that we choose independent of $x(0)$, since $N_{ij}$ and $x(0)$ are independent by construction. Together with  \eqref{x-rep-2} this entails  that $x$ can be represented (in distribution) as 
\begin{equation}\label{x-rep-3}
x(t)=x(0)+\sum_{i,j: \,i\neq j} (j-i)N^0_{ij}\left(\int_0^t g_{ij}(s)ds\right).
\end{equation}
Using the fact that $t\mapsto N^0_{ij}(t)$ is a.s. increasing (indeed, by stationarity of the Poisson process we have for $s\le t$, $P( N^0_{ij}(t)- N^0_{ij}(s)\ge 0)=P( N^0_{ij}(t-s)\ge 0)=1$) and  Cauchy-Schwarz inequality we have 
$$
E[e^{\frac{\a}{2} |x|_T}]\le E[e^{\a x(0)}]^{1/2}E\left[\exp{\a \sum_{i,j: \,i\neq j} |j-i|N^0_{ij}\left(\int_0^T g_{ij}(s)ds\right)}\right]^{1/2}.
$$
Using the explicit form of the moment generating function of each of the independent Poisson processes, we obtain
$$
E\left[\exp{\a \sum_{i,j: \,i\neq j} |j-i|N^0_{ij}\left(\int_0^T g_{ij}(s)ds\right)}\right]=\exp{\sum_{i,j: \,i\neq j}\int_0^T \left(e^{\a|j-i|}-1\right)g_{ij}(s)ds}.
$$
Therefore, 
\begin{equation}\label{exp-3}
E[e^{\frac{\a}{2} |x|_T}]\le E[e^{\a x(0)}]^{1/2}\exp{\frac{1}{2}\sum_{i,j: \,i\neq j}\int_0^T \left(e^{\a|j-i|}-1\right)g_{ij}(s)ds}:=\kappa_0.
\end{equation}
\end{proof}

\subsection{Probability measures on $I$}

\medskip
Let $\mathcal{P}(I)$ denote the set of probability measures on $I$. For $\mu,\nu\in\mathcal{P}(I)$, the total variation distance is defined by the formula
\begin{equation}\label{TV-B}
d(\mu,\nu)=2\sup_{A\subset I}|\mu(A)-\nu(A)|=\underset{i\in I}\sum|\mu(\{i\})-\nu(\{i\})|.
\end{equation}
Furthermore, let $\mathcal{P}(\Omega)$ be the set of probability measures $P$ on $\Om$ and
$\mathcal{P}_2(\Omega)$ be the subset of probability measures $P$ on $\Om$ such that 
$$
\|P\|_2^2:=\int_{\Om}|w|^2_TP(dw)=E[|x|_T^2]<+ \infty,
$$
 where $|x|_t:=\underset{0\le s\le t}\sup |x(s)|,\,\,\,0\le t\le T$.
 \ms\no Define on $\F$ the total variation metric
\begin{equation}\label{TV-F}
d(P,Q):=2\underset{A\in\F}\sup|P(A)-Q(A)|.
\end{equation}
Similarly, on the filtration $\Ff$, we define the total variation metric between two probability measures $P$ and $Q$ as
\begin{equation}\label{TV-filt}
D_t(P,Q):=2\underset{A\in\F_t}\sup|P(A)-Q(A)|,\quad 0\le t\le T.
\end{equation}
It satisfies 
\begin{equation}\label{ordering}
D_s(P,Q)\le D_t(P,Q),\quad 0\le s\le t.
\end{equation}
 For $P, Q\in \mathcal{P}(\Om)$ with time marginals $P_t:=P\circ x^{-1}(t)$ and $Q_t:=Q\circ x^{-1}(t)$, the total variation distance between $P_t$ and $Q_t$ satisfies
\begin{equation}\label{margine}
d(P_t,Q_t)\le D_t(P,Q),\quad 0\le t\le T.
\end{equation}
Indeed, we have
\begin{equation*}\begin{array}{lll}
d(P_t,Q_t):=2\underset{B\subset I}\sup|P_t(B)-Q_t(B)|=2\underset{B\subset I}\sup|P(x^{-1}(t)(B))-Q(x^{-1}(t)(B))|\\\qquad\quad\quad\;\; \le 2\underset{A\in\mathcal{F}_t}\sup|P(A)-Q(A)|=D_t(P,Q).
\end{array}
\end{equation*}
 Endowed with the total variation metric $D_T$, $\mathcal{P}(\Om)$ is a complete metric space. Moreover, $D_T$ carries out the usual topology of weak convergence. But, $(D_T,\mathcal{P}_2(\Om))$ may not be complete simply because the total variation metric does not guarantee existence of finite moments. This makes this distance less suitable for the study of models of mean-field type where the mean-field interaction is of the type $E[x(t)]$ or $E[\varphi(x(t))]$ when $\varphi$ is a Lipschitz function. Nevertheless, as we show it below, the following subset of $\mathcal{P}_2(\Om)$
$$
 \mathbb{D}_{p,\kappa}:=\{Q\in \mathcal{P}_2(\Omega),\, dQ=XdP,\,\, X \,\text{is\,} \mathcal{F}_T\text{-measurable and \,\,} E[X^p]\leq \kappa\},
 $$
where $p>1$ and $\kappa>0$ are given constants, which fits with our framework, turns out a complete metric space when endowed with the total variation norm $D_T$. 
Using this space requires a higher degree of smoothness on the intensity matrix of the Markov chain we impose below.

%%%%%%%%%%%%%%%%
\section{Jump processes of mean-field type}
In this section we prove existence of a unique probability measure $\wt P$ on $(\Omega, \F)$ under which the coordinate process $x$ is a jump process with intensities $\lambda_{ij}(t,x,{\wt P}\circ x^{-1}(t)),\,\, i,j\in I$, where we allow the jump intensities at time $t$ depend on the whole path $x$ over the time interval $[0,T]$ and also on the marginal law of $x(t)$, as long as the intensities are predictable. Because of the dependence of its jump intensities on the marginal law, we  call it   {\it jump process of mean-field type}.  If the intensities are deterministic functions of $t$ and the marginal law of $x(t)$ i.e. they are of the form $\lambda_{ij}(t, {\wt P}\circ x^{-1}(t)),\,\, i,j\in I$, we  call $x$  a {\it Markov chain of mean-field type} or simply an {\it nonlinear Markov chain}. 

The probability measure $\wt P$ is constructed as follows. We start with the probability measure $P$ which solves the martingale problem associated with $G=(g_{ij})$, where the intensities $g_{ij}$ are assumed time-independent, making the coordinate process $x$ a time-homogeneous Markov chain. Then, using a Girsanov-type change of measure in terms of a Dol{\'e}ans-Dade exponential martingale for jump processes which involves the intensities $\l_{ij}$ and $g_{ij}$, we obtain our probability measure $\wt P$. It is also possible to choose $G$ time-dependent. But, it is easier to deal with time-independent intensities.

 Let $\l$ be a measurable process from $[0,T]\times I\times I\times \Om\times \mathcal{P}(I)$ into $(0,+\infty)$ such that
 \begin{itemize}
 \item[(A1)] For every $Q\in \mathcal{P}_2(\Om)$, the process $((\l_{ij}(t, x,Q\circ x^{-1}(t)))_t$ is predictable.
 
 \item[(A2)] There exists a  positive constant $c_1$ such that for every $(t,i,j)\in[0,T]\times I\times I;\,i\neq j$, $w\in \Om$ and  $\mu, \nu \in\mathcal{P}(I)$,
 $$
 \l_{ij}(t,w,\mu)\ge c_1>0.
 $$
 
\item[(A3)] For $p=1,2,$ and for every for $t\in[0,T]$, $w\in \Om$ and  $\mu\in \mathcal{P}_2(I)$, 
 $$
 \underset{i,j: \, j\neq i}\sum |j-i|^p\l_{ij}(t,w,\mu)\le C(1+|w|^p_t+\int|y|^p\mu(dy)).
 $$
 \item [(A4)] The probability measure $\xi$ on $I$ has finite second moment: 
 $$
 \|\xi\|_2^2:=\int |y|^2\xi(dy)<\infty.
 $$
 
  \item[(A5)]  For $p=1,2,$ and for every $t\in[0,T]$, $w, \tilde w\in \Om$ and  $\mu, \nu \in\mathcal{P}(I)$,
  $$
  \underset{i,j: \, j\neq i}\sum |j-i|^p|\l_{ij}(t,w,\mu)-\l_{ij}(t,\tilde w,\nu)| \le C(|w-\tilde w|^p_t+d^p(\mu,\nu)).
  $$
 \end{itemize}
 
 \ms\no
\begin{remark}
\begin{enumerate}
\item Assumption (A4) is needed to guarantee that the chain has finite second moment.
\item Since $|j-i|\ge 1$, we obtain from (A5)  the following Lipschitz property of the intensity matrix
$$
\| \l(t,w,\mu)-\l(t,\tilde w,\nu)\| \le C(|w-\tilde w|_t+d(\mu,\nu)).
$$  
\end{enumerate}
 \end{remark}

 \ms\no 
 
\begin{example} {\bf A mean-field Schl\"ogl model}. In the mean-field  version of the Schl\"ogl model (cf. \cite{DZ}, \cite{FZ} and \cite{Fe}) the intensities are
 \begin{equation}\label{schlogl}
 \l_{ij}(w,\mu):=\left\{\begin{array}{ll} \nu_{ij} & \text{if} \,\, j\neq i+1,\\
 \nu_{i i+1}+\|\mu\|_1  & \text{if} \,\, j= i+1,
 \end{array}
 \right.
 \end{equation}
 where $\|\mu\|_1=\int |y|\mu(dy)$ is the first moment of the probability measure $\mu$ on $I$ and $(\nu_{ij})$ is a $Q$-matrix satisfying $\inf_{i\in I}\nu_{ii+1}>0$ and there exists $N_0\ge 1$ such that $\nu_{ij}=0,\,\,\text{for\,\,} |i-j|\ge N_0$. The martingale problem formulation states that, for every $f$ on $I$, the process defined by
\begin{equation*}\label{S-f-mart-1}
M_t^f:=f(x(t))-f(x(0))-\int_{(0,t]}(\wt G(s)f)(x(s))\,ds
\end{equation*}
is a local martingale relative to $(\Om,\F,\Ff)$, where 
\begin{equation}\label{S-G-f}
\wt G(s)f(i)=\sum_{j: \,j\neq i}\nu_{ij}(f(j)-f(i))+\sum_j jP(x(s)=j)(f(i+1)-f(i)).
\end{equation}
\end{example}
 
 Let $P$ be the probability under which $x$ is a time-homogeneous Markov chain such that $P\circ x^{-1}(0)=\xi$ and with time-independent $Q$-matrix $(g_{ij})_{ij}$ satisfying \eqref{G} and \eqref{exp-1}.  
 
 Assume further that
\begin{equation}\label{G2}
c_2:=\underset{i,j: i\neq j}\inf g_{ij}>0.
 \end{equation}
This condition is needed below  to obtain estimates involving the density of Girsanov-Dol{\'e}ans-Dade type change of measure between two probability measures under which the chain has jump intensities are $\l_{ij}$ and $g_{ij}$, respectively. This amounts to only taking into account nonzero jump intensities.

\ms\no
 To ease notation we set, for $(t,i,j)\in[0,T]\times I\times I$ and  $Q\in \mathcal{P}(\Om)$
 \begin{equation}\label{lambda-Q}
 \l^{Q}_{ij}(t):=\l_{ij}(t,x,Q\circ x^{-1}(t)),\quad    \l^{Q}_{ii}(t):=-\underset{j\neq i}\sum\,\l^{Q}_{ij}(t).
 \end{equation}
\noindent  Let $P^Q$ be the measure on $(\Om,\F)$ defined by 
 \begin{equation}\label{PQ}
 dP^Q:=L^Q(T) dP,
 \end{equation}
where 
\begin{equation}\label{exp-mg}
L^Q(T):=\underset{\substack{i,j\\ i\neq j}}\prod \exp{\left\{ \int_{(0,T]}\ln{\frac{ \l^{Q}_{ij}(t)}{g_{ij}}}\,dN_{ij}(t)-\int_0^T( \l^{Q}_{ij}(t)-g_{ij})I_i(t)dt  \right\}},
\end{equation}
is the Dolean-Dade exponential. It is the solution of the following linear stochastic integral equation
\begin{equation}\label{L-sde-1}
L^Q(t)=1+\int_{(0,t]} L^Q(s^-)\underset{i,j:\, i\neq j}\sum I_i(s^-)\ell^Q_{ij}(s)dM_{ij}(s),
\end{equation}
where
\begin{equation}\label{L-sde-2}
\ell^Q_{ij}(s)=\left\{\begin{array}{rl}
\l_{ij}^Q(s)/g_{ij}-1 &\text{if }\,\, i\neq j,\\ 0 & \text{if }\,\, i=j,
\end{array}
\right.
\end{equation}
and $(M_{ij})_{ij}$ is the $P$-martingale given in \eqref{mart-1}.

\ms\no If $L^Q$ is a $P$-martingale, then by Girsanov theorem,  $P^Q$ is a probability measure  on  $(\Om,\F)$ under which the coordinate process $x$ is a jump process with intensity matrix $\l^Q:=(\l^{Q}_{ij}(t))_{i,j}$ and starting distribution $P^Q\circ x^{-1}(0)=\xi$. In particular, the compensated processes associated with the counting processes $N_{ij}$ defined by
 \begin{equation}\label{mart-Q}
M^Q_{ij}(t):=N_{ij}(t)-\int_{(0,t]} I_i(s^-) \l^{Q}_{ij}(s)ds,\quad M^Q_{ij}(0)=0,
\end{equation}
are zero mean, square integrable and mutually orthogonal $P^Q$-martingales  whose  predictable quadratic variations are
\begin{equation}\label{mart-Q-2}
\langle M^Q_{ij}\rangle_t=\int_{(0,t]} I_i(s^-) \l^{Q}_{ij}(s)ds.
\end{equation} 
Using \eqref{L-sde-2}, we may write $M^Q_{ij}$ in terms of $M_{ij}$ as follows.
\begin{equation}\label{m-m-Q}
M^Q_{ij}(t)=M_{ij}(t)-\int_{(0,t]} \ell^{Q}_{ij}(s)I_i(s^-)g_{ij}ds.
\end{equation}
Now, since $L^Q$ is a positive $P$-local martingale, it is a supermartingale. Thus, $E[L^Q_T]\le 1$. In order to show that is a $P$-martingale, we need to show that $E[L^Q_T]=1$. We note that the imposed conditions (A1)-(A4) on the intensity matrix $\l^Q$ do not fit with the assumptions displayed in the literature ranging from \cite{bremaud}, Theorem T11, to \cite{sokol}, Theorem 2.4, to guarantee that $L^Q$ is a $P$-martingale.   

\ms\no To show that $L^Q$ is a $P$-martingale we will use the following apriori estimate.

\begin{lemma}\label{x-P-Q-esti} Let  $Q\in \mathcal{P}_2(\Omega)$ and assume $\l^Q$ and $\xi$ satisfy (A1)-(A4). If $P^Q$, given by \eqref{PQ}, is a probability measure  on  $(\Om,\F)$, then
\begin{equation}\label{x-esti-2}
\|P^Q\|_2^2=E_{P^Q}[|x|^2_T]\le Ce^{CT}(1+\|\xi\|_2^2+\|Q\|_2^2)<+\infty.
\end{equation}
In particular,
\begin{equation}\label{x-P-esti}
\|P\|_2^2=E[|x|^2_T]\le Ce^{CT}(1+\|\xi\|_2^2)<+\infty.
\end{equation}
\end{lemma}

\begin{proof} Since, by Girsanov theorem, under $P^Q$, $x$ has jump intensity $\l_{ij}(t,x,Q\circ x^{-1}(t))$, applying a similar formula as \eqref{f-mart-1} to $f(x)=x$, where instead of $G$ we use the matrix $\l^Q$, we obtain
\begin{equation}\label{x-mart-1}
x(t)=x(0)+\int_{(0,t]}\l^Q(s)x(s)\,ds+M_t^x,
\end{equation}
where 
\begin{equation}\label{x-mart-2}
M^x(t)=\int_{(0,t]}Z^x(s)dM(s)=\underset{i,j: \, j\neq i}\sum \,(j-i)M_{ij}(t),
\end{equation}
with 
\begin{equation*}\label{x-mart-3}
\|Z^x(s)\|^2_{\l^Q}=\underset{i,j: \,j\neq i}\sum (j-i)^2I_i(s^-)\l^Q_{ij}(s),
\end{equation*}
and 
\begin{equation*}
\l^Q(s)x(s)=\underset{i,j: \,j\neq i}\sum(j-i)I_i(s^-)\l^Q_{ij}(s),
\end{equation*}
which, in view of (A3), satisfy
\begin{equation}\label{G-x}
|\l^Q(s)x(s)|^2+\|Z^x(s)\|^2_{\l^Q}\le C\big(1+|x|^2_s+\int w_s^2Q(dw)\big),\,\, 0\le s\le T.
\end{equation}
Therefore, applying the Cauchy-Schwarz inequality together with \eqref{Doob} to \eqref{x-mart-1} we obtain
\begin{equation*}
E_{P^Q}[|x|^2_T]\le C E_{P^Q}\left[|x(0)|^2+\int_{(0,T]} \left(|\l^Q(s)x(s)|^2+\|Z^x(s)\|^2_{\l^Q}\right)\,ds\right],
\end{equation*}
and by \eqref{G-x} we get
\begin{equation}\label{x-esti}
E_{P^Q}[|x|^2_T]\le C\Big(1+\|\xi\|_2^2+\int_{(0,T]}\left(E_{P^Q}[|x|^2_s]+\int w_s^2Q(dw)\right)\,ds\Big).
\end{equation}
Using $\int w_s^2Q(dw)\le \int |w|_s^2Q(dw)\le \|Q\|_2^2$ and applying Gronwall's inequality we finally get \eqref{x-esti-2}.
\end{proof}

\begin{proposition}\label{L-Q-mart} Let  $Q\in \mathcal{P}_2(\Omega)$ and assume $\l^Q$ and $\xi$ satisfy (A1)-(A4). Then, $L^Q$ is a $P$-martingale.
\end{proposition}
\begin{proof}  The proof is inspired by the proof of Proposition (A.1) in \cite{EH}.  As mentioned above, it suffices to prove that $E[L^Q_T]=1$. For $n\ge 0$, let $\l^n$ be the predictable intensity matrix given by $\l^n_{ij}(t):=\l^Q_{ij}(t)\mathbf{1}_{\{\omega,\,\,|x(\omega)|_{t^-}\le n\}}$ and let $L^n$ be the associated Dolean-Dade exponential and $P^n$ the positive measure defined by $dP^n=L^n_T dP$.  Noting that, for $i, j\in I, \,i\neq j$, $|i-j|\ge 1$, by (A3), we have
$$
\l_{ij}(t,w,\mu)\le C(1+|w|_t+\int|y|\mu(dy)).
$$
Thus, for every $n\ge 1$, $\l^n_{ij}(t)\le C(1+n+\|Q\|_2)$, i.e. $\l^n_{ij}$ is bounded. In view of \cite{bremaud}, Theorem T11, $L^n$ is a $P$-martingale. In particular, $E[L^n_T]=1$ and $P^n$ is a probability measure. By \eqref{x-P-esti}, $|x|_T<\infty,\, P$-a.s. Therefore, on the set $\{\omega,\,\,|x(\omega)|_T\le n_0\}$, for all $n\ge n_0$, $L^n_T(\omega)=L_T^Q(\omega)$. This in turn yields that 
$L^n_T\to L_T^Q,\,\,P$-a.s., as $n\to +\infty$. Now, if $(L^n_T)_{n\ge 1}$ is uniformly integrable, the $P$-a.s. convergence implies $L^1(P)$-convergence of $L^n_T$ to $L_T^Q$, yielding  $E[L^Q_T]=1$. It remains to show that $(L^n_T)_{n\ge 1}$ is uniformly integrable: 
$$
\lim_{a\to \infty}\sup_{n\ge 1}\int_{\{L^n_T>a\}} L^n_T\,dP=0.
$$
For $m\ge 1$, set $\t_m=\inf\{t\le T,\,\, |x|_t\ge m\}$ if the set is nonempty and $\t_m=T+1$ if it is empty. Denoting by $E^n$ the expectation w.r.t. $P^n$, we have
\begin{equation}\label{tau-m}\begin{array}{lll}
\int_{\{\t_m\le T\}} L^n_T\,dP=P^n(\t_m\le T)=P^n(|x|_T\ge m)\\ \qquad\qquad\qquad\quad \le E^n[|x|_T]/m\le C/m,
\end{array}
\end{equation}
where, by \eqref{x-esti-2}, $C$ does not depend on $n$.  

\ms\no Let $\e>0$. Choose $m_0\ge 1$ such that $C/m_0<\e$. We have, for all $n\ge m_0$, $L^n_{T\wedge \t_{m_0}}=L^{m_0}_{T\wedge \t_{m_0}}$. This entails that
\begin{equation*}
\sup_{n\ge 1}\int_{\{L^n_{T\wedge \t_{m_0}}>a\}} L^n_{T\wedge \t_{m_0}}\,dP= \max_{n\le m_0}\int_{\{L^n_{T\wedge \t_{m_0}}>a\}} L^n_{T\wedge \t_{m_0}}\,dP \to 0,\,\, a\to  \infty.
\end{equation*}
So there exists $a_0>0$ such that whenever $a>a_0$, 
\begin{equation}\label{L-m}
\max_{n\le m_0}\int_{\{L^n_{T\wedge \t_{m_0}}>a\}} L^n_{T\wedge \t_{m_0}}\,dP< \e.
\end{equation}
We have
$$
\begin{array}{lll}
\underset{n\ge 1}\sup \int_{\{L^n_T>a\}} L^n_T\,dP\le\underset{n\ge 1}\sup\int_{\{L^n_T>a,\, \t_{m_0}\le T\}} L^n_T\,dP+\underset{n\ge 1}\sup\int_{\{L^n_T>a,\, \t_{m_0}> T\}} L^n_T\,dP \\ \qquad\qquad\qquad\qquad \le \underset{n\ge 1}\sup\int_{\{\t_{m_0}\le T\}} L^n_T\,dP+\underset{n\ge 1}\sup\int_{\{L^n_{T\wedge \t_{m_0}}>a,\, \t_{m_0}> T\}} L^n_{T\wedge \t_{m_0}}\,dP \\ \qquad\qquad\qquad\qquad \le \underset{n\ge 1}\sup\int_{\{\t_{m_0}\le T\}} L^n_T\,dP+\underset{n\le m_0 }\max\int_{\{L^n_{T\wedge \t_{m_0}}>a\}} L^n_{T\wedge \t_{m_0}}\,dP\\
\qquad\qquad\qquad\qquad\le C/m_0+\e < 2\e,
\end{array}
$$
 in view of \eqref{tau-m} and \eqref{L-m}. This finishes the proof since $\e$ is arbitrary.
\end{proof}

\ms\noindent Next, we will show that there is $\widehat Q$ such that $P^{\widehat Q}={\widehat Q}$, i.e., $\widehat Q$ is a fixed point. It is the probability measure under which the coordinate process is a jump process of mean-field type.  

\begin{lemma}\label{subspace} Let $p>1$ and $\kappa>0$ be given constants. The set $\mathbb{D}_{p,\kappa}$ defined by  
$$
 \mathbb{D}_{p,\kappa}:=\{Q\in \mathcal{P}_2(\Omega),\, dQ=XdP,\,\, X \,\text{is\,} \mathcal{F}_T\text{-measurable and \,\,} E[X^p]\leq \kappa\},
 $$
endowed withe total variation norm $D_T$ is a complete metric space.
 \end{lemma}
 \begin{proof} Let $(Q_n)_{n\geq 0}$ be a Cauchy sequence in $(\mathbb{D}_{p,\kappa},D_T)$. Thus, for any $n\geq 0$, $dQ^n=X_ndP$ with $E[X_n^p]\leq \kappa$. Since $(\mathcal{P}(\Omega), D_T)$ is a complete metric space and $\mathbb{D}_{p,\kappa}\subset \mathcal{P}(\Omega)$, there exists a probability $Q\in \mathcal{P}(\Omega)$ such that $D_T(Q_n,Q)\to 0$ as $n\to\infty$. Next, since $p>1$, there exists a subsequence, which we still denote by $(X_n)_{n\geq 0}$, and an $\mathcal{F}_T$-measurable random variable $X$ such that  $X_n\rightarrow X$ weakly in $L^p(\Omega,\mathcal{F}_T,P)$, as $n\to\infty$. But, for any $A\in \mathcal{F}_T$
 $$
 \lim_nQ^n(A)=\lim_nE[X_n1_A]=E[X1_A]=Q(A).
 $$
This entails that the probability $Q$ has a density w.r.t. $P$ which is given by $X$. On the other hand, by the semi-continuity of the norm w.r.t. the weak topology, we obtain $E[X^p]\leq \liminf_nE[X_n^p]\leq \kappa.$ Finally, by H\"older's inequality, we have 
$$
E_Q[|x|_T^2]=E[|x|_T^2X]\leq (E[X^p])^{\frac{1}{p}}(E[|x|_T^{2q}])^{\frac{1}{p}}<+\infty
$$ 
for $q$ such that $\frac{1}{p}+\frac{1}{q}=1$, since, by \eqref{exp-1},  $|x|_T$ satisfies (\ref{q-norm}). Hence, $Q$ belongs to  $ \mathbb{D}_{p,\kappa}$ and $(\mathbb{D}_{p,\kappa},D_T)$ is a complete metric space. 
\end{proof}

 \medskip
\begin{theorem}\label{FP-tv} Assume (A1)-(A5) and consider the metric space $(\mathcal{P}(\Om),D_T)$. Then, the map 
\begin{equation*}\begin{array}{lll}
\Phi: \mathcal{P}_2(\Om)\longrightarrow \mathcal{P}_2(\Om) \\ 
\qquad\quad Q  \mapsto \Phi(Q):=P^Q;\quad   dP^Q:=L^Q(T) dP,\,\,\, P^Q\circ x^{-1}(0)= \xi,
\end{array}
\end{equation*} 
is well defined. 

\ms\no
Assume further that the intensities $\lambda_{ij}$ satisfy the following condition.
\begin{itemize}
\item[(A6)]  For every for $t\in[0,T]$, $w\in \Om$ and  $\mu\in \mathcal{P}_1(I)$, 
 $$
 \underset{i,j: \, j\neq i}\sum \l^2_{ij}(t,w,\mu)\le C(1+|w|_t+\int |y|\mu(dy)).
 $$
 \item[(A7)] There exits $\g\ge \frac{2T}{c_2}$ such that $\int e^{\g y}\xi(dy)<\infty$.
 \end{itemize}
 
Then, the probability density $L^Q_T$ is bounded in $L^2(P)$. Moreover, there exists a constant $\kappa>0$ such that $\Phi(\mathbb{D}_{2,\kappa})\subset \mathbb{D}_{2,\kappa}$. Furthermore, $\Phi$ admits a fixed point in $\mathbb{D}_{2,\kappa}$.

If $\widehat{Q}$ denotes such a fixed point, it satisfies
\begin{equation}\label{x-Q-bar}
\|\widehat Q\|_2^2\le Ce^{2CT}(1+\|\xi\|_2^2)<\infty.
\end{equation}
\end{theorem} 
\begin{proof}
Let  $Q\in \mathcal{P}_2(\Omega)$. Then, by \eqref{x-esti-2}, 
\begin{equation*}
\|\Phi(Q)\|_2^2=E_{\Phi(Q)}[|x|^2_T]\le Ce^{CT}(1+\|\xi\|_2^2+\|Q\|_2^2)<+\infty,
\end{equation*}
which implies that $\Phi(Q)\in \mathcal{P}_2(\Om)$, since $\|\xi\|_2^2<+\infty$ by (A4).

\ms
Next, we show that $E[(L^Q_T)^{2}]\le \kappa$ for some constant $\kappa>0$. Let 
$h_{ij}(t):=\lambda^Q_{ij}(t)/g_{ij}$
and define
$$
\varphi_T:=\underset{\substack{i,j\\ i\neq j}}\prod \exp{\left\{ \int_{(0,T]}\ln{h^{2}_{ij}}\, dN_{ij}(t)-\int_0^T\left( \frac{h^{4}_{ij}(t)}{2}-\frac{1}{2}\right)g_{ij}I_i(t)dt  \right\}}
$$
and
$$
\psi_T:=\underset{\substack{i,j\\ i\neq j}}\prod \exp{\left\{ \int_0^T \left( \frac{h^{4}_{ij}(t)}{2}-\frac{1}{2}-2(h_{ij}(t)-1) \right)g_{ij}I_i(t)dt \right\}}.
$$
We have
$$
(L^Q_T)^{2}=\varphi_T\psi_T.
$$
The choice of $\varphi_T$ is made so that the process $\varphi^{2}_t$ is a Doleans-Dade  positive supermartingale satisfying $E[\varphi^{2}_T]\le 1$. Furthermore, by \eqref{G}, \eqref{G2} and (A6), we have
$$\begin{array}{lll}
\psi_T(\a)\le\underset{\substack{i,j\\ i\neq j}}\prod \exp{\left\{ \int_0^T \frac{(\lambda^Q_{ij}(t))^2}{2g_{ij}}+\frac{3}{2}g_{ij} \right\}}\\ \qquad\quad \le \exp{\left\{ \frac{CT}{2c_2}(1+|x|_T+\|Q\|_1)+\frac{3T}{2}\underset{\substack{i,j:\,i\neq j}}\sum g_{ij} \right\}}. 
\end{array}
$$
Therefore,
\begin{equation}\label{psi-est}
E[\psi^{2}_t(\a)]^{1/2}\le \mu_TE\left[e^{\frac{T}{c_2}|x|_T}  \right]^{1/2},
\end{equation}
where 
$$
\mu_T:=exp{\left\{T\left[\frac{C}{c_2}(1+\|Q\|_1)+3\underset{\substack{i,j:\,i\neq j}}\sum g_{ij}\right]\right\}}.
$$
Since, by H\"older inequality, we have 
$$ 
E[(L^Q_T)^2]\le E[\varphi^{2}_T]^{1/2}E[\psi^{2}_T]^{1/2}\le E[\psi^{2}_T]^{1/2},
$$
It follows from \eqref{psi-est} that
$$
E[(L^Q_T)^{2}]\le \mu_TE\left[e^{\frac{T}{c_2}|x|_T} \right]^{1/2}.
$$
Setting $\a:=2\frac{T}{c_2}$, in view of \eqref{exp-3}, we obtain
$$
E[(L^Q_T)^{2}]\le \mu_T\kappa_0:=\kappa.
$$

\ms\no
Next, we show the contraction property of the map $ \Phi$ on $\mathbb{D}_{2,\kappa}$. To this end, given $Q,\wt{Q}\in\mathcal{P}_2(\Om)$, we use an estimate of the total variation distance $D_T(\Phi(Q),\Phi(\wt{Q}))$ in terms of the relative entropy $H(\Phi(Q)|\Phi({\wt Q)})$ between  $\Phi(Q)$ and $\Phi(\wt{Q})$ given by the celebrated Csisz{\'a}r-Kullback-Pinsker inequality:
\begin{equation}\label{CKP}
D^2_T(\Phi(Q),\Phi(\wt{Q}))\le 2H(\Phi(Q)|\Phi({\wt Q)}),
\end{equation}
where, 
$$
H(\Phi(Q)|\Phi({\wt Q)})=E_{\Phi(Q)}\Big[\log\frac{L^Q(T)}{L^{\wt Q}(T)}\Big].
$$
In view of \eqref{exp-mg}, we have
$$
\log\frac{L^Q(T)}{L^{\wt Q}(T)}=\sum_{i,j:\, i\neq j} \int_{(0,T]}\ln{\frac{ \l^{Q}_{ij}(t)}{\l^{\wt Q}_{ij}(t)}}dN_{ij}(t)-\int_0^T( \l^{Q}_{ij}(t)-\l^{\wt Q}_{ij}(t))I_i(t)dt.
$$
Taking expectation w.r.t. $\Phi(Q)$, using \eqref{mart-Q}, we obtain 
\begin{eqnarray}\label{entropy}
H(\Phi(Q)|\Phi({\wt Q)})=\sum_{i,j:\, i\neq j}E_{\Phi(Q)}\left[\int_0^T \left(\tau(\l^{Q}_{ij}(t))-\tau(\l^{\wt Q}_{ij}(t)) \right. \right. \\
 \left. \vphantom{\int_0^T}\left. -(\l^{Q}_{ij}(t)-\l^{\wt Q}_{ij}(t))\log\l^{\wt Q}_{ij}(t)\right)\,I_i(t)dt\right], \nonumber
\end{eqnarray}
where
$\tau(x):=x\log x -x+1,\,\, x>0$ is a convex function. We note that the r.h.s. of this last equality is non-negative, since, by convexity, we have
$$
\tau(x)\ge \tau(y)+(x-y)\tau^{\prime}(y),
$$
where $\tau^{\prime}(y)=\log y$. Using Taylor expansion we get
$$
\tau(x)=\tau(y)+(x-y)\tau^{\prime}(y)+\frac{1}{2}(x-y)^2\tau^{\prime\prime}(z),
$$
for some $z\in\{tx+(1-t)y:\,\,t\in (0,1)\}$, where $\tau^{\prime\prime}(z)=1/z$. Taking $x,y$ such that $ x,y\ge c_1>0$, as in (A2), we obtain
\begin{equation}\label{convex}
\tau(x)\le\tau(y)+(x-y)\tau^{\prime}(y)+\frac{1}{2c_1}(x-y)^2.
\end{equation}
Applying \eqref{convex} to the entropy \eqref{entropy}, we obtain
$$
H(\Phi(Q)|\Phi({\wt Q)})\le \frac{1}{2c_1}\sum_{i,j:\, i\neq j}E_{\Phi(Q)}\left[\int_0^T (\l^{Q}_{ij}(t)-\l^{\wt Q}_{ij}(t))^2 \,dt\right].
$$
Combining this inequality with \eqref{CKP}, we obtain

\begin{equation}\label{entropy-estimate}
D^2_T(\Phi(Q),\Phi(\wt{Q}))\le \frac{1}{c_1}\sum_{i,j:\, i\neq j}E_{\Phi(Q)}\left[\int_0^T (\l^{Q}_{ij}(t)-\l^{\wt Q}_{ij}(t))^2 \,dt\right].
\end{equation}
We may use (A5) to obtain
$$
\sum_{i,j:\, i\neq j}E_{\Phi(Q)}\left[(\l^{Q}_{ij}(t)-\l^{\wt Q}_{ij}(t))^2\right]\le C d^2(Q_t,{\wt Q}_t)\le C D_t^2(Q,{\wt Q}).
$$
Therefore,
$$
D^2_T(\Phi(Q),\Phi({\wt Q)})\le 2H(\Phi(Q)|\Phi({\wt Q)})\le \frac{C}{c_1}\int_0^T D_t^2(Q,{\wt Q}) \,dt.
$$
Iterating this inequality, we obtain, for every $N>0$,
$$
D^2_T(\Phi^N(Q),\Phi^N(\wt{Q}))\le C^N\int_0^T\frac{(T-t)^{N-1}}{(N-1)!}D^2_t(Q,\wt{Q})dt\le \frac{C^NT^N}{N!}D^2_T(Q,\wt{Q}),
$$
where $\Phi^N$ denotes the $N$-fold composition of the map $\Phi$. Hence, for $N$ large enough, $\Phi^N$ is a contraction which implies that $\Phi$ admits a unique fixed point.  

\ms\no Finally, using  \eqref{x-esti} with  $\Phi(\widehat Q)=\widehat Q$, noting that $ \int w_s^2\widehat Q(dw)\le \int |w|_s^2\widehat Q(dw)=E_{\widehat Q}[|x|^2_s]$,  we get
$$
E_{\widehat Q}[|x|^2_T]\le C\Big(1+\|\xi\|_2^2+2\int_{(0,T]}E_{\widehat Q}[|x|^2_s]\,ds\Big).
$$
Applying Gronwall's inequality we obtain the estimate
\begin{equation*}
\|\widehat Q\|_2^2\le Ce^{2CT}(1+\|\xi\|_2^2)<\infty.
\end{equation*}
\end{proof}
\begin{remark}
The mean-field Schl\"ogl model \eqref{schlogl} satisfies (A6) and (A7).
\end{remark}

\begin{corollary}
The mapping $t\mapsto P\circ x^{-1}(t)$ is continuous. More precisely, we have
\begin{equation}\label{marg-cont}
d(P_t,P_s)\le C(1+\|P\|_2)(t-s),\quad 0\le s\le t\le T.
\end{equation}

\begin{proof}
The inequality \eqref{marg-cont} follows by applying the above estimates to the martingale \eqref{f-mart-1} with $f(x)=I_{\{x\in A\}},\,\,\, A\subset I$, where we use the matrix $\l$ instead of $G$.
\end{proof}
\end{corollary}

\subsection{Markov chain BSDEs}
An important consequence of Theorem \eqref{mart-rep} are solutions $(Y,Z)$ of Markov chain backward stochastic differential equations (BSDEs) defined on $(\Omega, \mathcal{F},\mathbb{F},P)$ by
\begin{equation}\label{Y-bsde1}
-dY(t)=f(t,\omega,Y(t^-),Z(t))dt-Z(t)dM(t),\quad Y(T)=\zeta.
\end{equation}
It is easily seen that if $(Y,Z)$ solves \eqref{Y-bsde1} then it admits the following representation:
\begin{equation*}
Y(t)=E\Big[\zeta+\int_t^T f(s,\omega,Y(s^-), Z(s))\,ds\Big|\mathcal{F}_t\Big], \quad t\in[0,T].
\end{equation*}
Moreover, $t\mapsto Y(t)$ is right-continuous with left limits. Therefore,  $Y(t^-)=Y(t)\,, dt$-a.e. Hence, we may write
\begin{equation*}
Y(t)=E\Big[\zeta+\int_t^T f(s,\omega,Y(s), Z(s))\,ds\Big|\mathcal{F}_t\Big], \quad t\in[0,T].
\end{equation*}

Existence and uniqueness results of solutions of Markov chain BDSEs \eqref{Y-bsde1} based on the martingale representation theorem ($L^2$-theory) have been recently studied in a series of papers by Cohen and Elliott (see e.g. \cite{Cohen2012} and the references therein). Their approach essentially adapts the method for solving Brownian motion driven BSDEs established first in \cite{pardoux}. Recently, Confortola {\it et al.} \cite{conf} derived existence and uniqueness results for more general classes of BSDEs driven by marked point processes under only $L^1$-integrability conditions. In this paper we use the $L^2$-theory as we want to use the martingale representation theorem in our optimal control problem. 
 
Below, we establish existence of an optimal control and a saddle-point for the zero-sum game using some properties of the following class of BSDEs which is a special case of \eqref{Y-bsde1}.
\begin{equation}\label{Y-bsde2}
-dY(t)=\phi(t,x,Z(t))dt-Z(t)dM(t),\quad Y(T)=\zeta,
\end{equation}
$(M_{ij})_{ij}$ is the $P$-martingale given in \eqref{mart-1} and the driver $\phi$  is essentially of the form 
\begin{equation}\label{drift}
\phi(t,x,p):=f(t,x)+\langle \ell(t,x),p\rangle_g,
\end{equation}
where $p:=(p_{ij},\, i,j \in I)$ is a real-valued matrix indexed by $I\times I$ and $\ell=(\ell_{ij},\, i,j \in I)$ is given by
\begin{equation}\label{L-bsde}
\ell_{ij}(s,x):=\left\{\begin{array}{rl}
\l_{ij}(s,x)/g_{ij}-1 &\text{if }\,\, i\neq j,\\ 0 & \text{if }\,\, i=j,
\end{array}
\right.
\end{equation}
where the predictable process $\l(t,x)=(\l_{ij}(t,x),\, i,j \in I)$ is the intensity matrix of the chain $x$ under a probability measure $\wt P$ on $(\Om,\F)$ given by a similar formula as \eqref{PQ}-\eqref{exp-mg}. In particular, as in \eqref{m-m-Q}, the  processes 
\begin{equation}\label{m-m-bsde}
\wt M_{ij}(t)=M_{ij}(t)-\int_{(0,t]} \ell_{ij}(s,x)I_i(s^-)g_{ij}ds
\end{equation}
are zero mean, square integrable and mutually orthogonal $\wt P$-martingales  whose  predictable quadratic variations are
\begin{equation}\label{mart-bsde}
\langle \wt M_{ij}\rangle_t=\int_{(0,t]} I_i(s^-) \l_{ij}(s,x)ds.
\end{equation} 
Moreover, $\phi$ satisfies a 'stochastic' Lipschitz condition. More precisely, we make the following  assumptions on the driver $\phi$ and the terminal value $\xi$.

\begin{enumerate}
\item[ (H1)] $P$-a.s., for all $(t,\omega)\in[0,T]\times\Omega,~ z^1=(z^1_{ij}),z_2 =(z^2_{ij}),\,\, z^1_{ij}, z^2_{ij}\in \R$,
$$
|\phi(t,\omega, z_1)-\phi(s,\omega,z_2)|\le a(t)\|z^1-z^2\|_{g}(t),
$$
(see Notation \ref{g-t}) where $(a(t))_{t}$ is a nonnegative and progressively measurable process which belongs to $L^2([0,T]\times \Omega,dt\otimes dP)$. 
\item[(H2)]  $\phi(t,\omega,0)$ is bounded.
\item[(H3)] $\zeta$ is an $\mathcal{F}_T$-measurable and bounded random variable.
\end{enumerate}

\ms\no
First we establish a comparison result for solutions of the BSDE \eqref{Y-bsde2}. This  result is a key argument in the proof of existence and uniqueness of solutions of our BSDE.
\begin{proposition}\label{bsde-comp}
 Let, for $i=1,2$, $(Y^i,Z^i)$ be the solutions of the BSDE \eqref{Y-bsde2} associated with $(\phi^i,\zeta^i)$  respectively, where $(\phi^i,\zeta^i)$ satisfy (H1) to (H3). Assume that
\begin{enumerate}
\item[(H4)] $\zeta^1\ge \zeta^2, \quad P$-a.s.;
\item[(H5)] for any $(t,\omega)\in[0,T]\times\Omega,\,$ $\phi^1(t,\omega,Z^2(t))\ge \phi^2(t,\omega, Z^2(t))  \quad P\text{-a.s.}$
 \end{enumerate}
 then, $Y^1\ge Y^2$ on $[0,T]\quad P$-a.s.
\end{proposition}

\begin{proof}
Set
$$
\widehat Y:=Y^1-Y^2,\quad \widehat Z:=Z^1-Z^2,\quad \widehat\zeta:=\zeta^1-\zeta^2.
$$
Using (H4) and (H5), we obtain
$$
\begin{array}{lll}
\widehat Y_t=\widehat\zeta+\int_t^T (\phi^1(s,x,Z^1(s))-\phi^1(s,x,Z^2(s)))ds\\ \quad +\int_t^T (\phi^1(s,x,Z^2(s))-\phi^2(s,x,Z^2(s)))ds -\int_t^T\widehat Z(s)dM(s)\\ \quad \ge \int_t^T (\phi^1(s,x,Z^1(s))-\phi^1(s,x,Z^2(s)))ds-\int_t^T\widehat Z(s)dM(s) \quad P\text{-}\as
\end{array}
$$
Since
$$
\phi^1(s,x,Z^1(s))-\phi^1(s,x,Z^2(s))=\langle \ell(s,x),\widehat Z(s)\rangle_g,
$$
by \eqref{m-m-bsde} 
$$
\int_0^t\widehat Z(s)d\wt M(s)
=\int_0^t\widehat Z(s)dM(s)-\int_0^t\langle \ell(s,x),\widehat Z(s)\rangle_g ds
$$
is a zero-mean $\wt P$-martingale. Since $\wt P$ is absolutely continuous w.r.t. $P$, we also have
$$
\widehat Y_t\ge -\int_t^T\widehat Z(s)d\wt M(s) \quad \wt P\text{-a.s.}
$$
Hence, taking conditional expectation  w.r.t. $\F_t$, we obtain
$$
\widehat Y_t\ge -\wt E\left[\int_t^T\widehat Z(s)d\wt M(s)|\F_t \right]=0 \quad P\,\,\text{and}\,\, \wt P\text{-a.s.}
$$

\end{proof}

\begin{theorem}\label{bsde-th}
Let  $(\phi,\zeta)$ satisfies the assumptions (H1) to (H3). Then, the BSDE \eqref{Y-bsde2} associated with $(\phi,\zeta)$ admits a solution $(Y,Z)$ consisting of an adapted process $Y$ which is right-continuous with left limits and a predictable process $Z$ which satisfy  
\begin{equation*}
E\left[\sup_{t\in[0,T]}|Y(t)|^2+\int_{(0,T]} \|Z(s)\|^2_{g}ds\right]<+\infty.
\end{equation*}  
This solution is unique up to indistinguishability for $Y$ and for $Z$ equality  $dP\times g_{ij}(s^-)I_i(s^-)ds$-almost everywhere. 

\end{theorem}

Using Proposition \ref{bsde-comp}, the proof of the theorem is similar to that of the Brownian motion driven BSDEs derived in \cite{Ham-Lepl95}, Theorem I-3, using an approximation scheme by an increasing sequence of standard Markov chain BSDEs for which existence, uniqueness and comparison results are similar to that of the Brownian motion driven BSDEs derived in \cite{pardoux} and \cite{EPQ}, along with the properties \eqref{isometry} and \eqref{Doob} related to the martingale $W$ displayed in \eqref{stoch} together with It\^o's formula for semimartingales driven by counting processes. We omit the details.

%%%%%%%%%%%%%%%%%%%%%%%%%%%%%%%%%%%%%%%%%%%%%%%%%%%%%%%%%%%%%%%%%%%%%%%

\section{Optimal control of jump processes of mean-field type}\label{control}

In this section we  perform a detailed study of the control problem using the total variation distance as a carrier of the topology of weak convergence. 

\ms Let $(U, \delta)$ be a compact metric space with its Borel field $\mathcal{B}(U)$ and $\U$ the set of $\Ff$-progressively measurable processes $u=(u(t),\,0\le t\le T)$ with values in $U$. We call $\U$ the set of admissible controls. In this section we consider a control problem of the jump process of mean-field type introduced above, where the control enters the jump intensities. 

\ms\no
For $u\in\U$, let $P^u$ be the probability measure on $(\Om,\F)$ under which the coordinate process $x$ is a jump process with intensities 
\begin{equation}\label{u-lambda}
\l_{ij}^u(t):=\l_{ij}(t,x,P^u\circ x^{-1}(t),u(t)),\,\,\, i,j\in I,\,\, 0\le t\le T,
\end{equation}
satisfying the following assumptions similar to (A1)-(A5). 
\begin{itemize}
 \item[(B1)] For any $u\in\U$, $i,j\in I$,  the process $((\l_{ij}(t, x,P^u\circ x^{-1}(t),u(t)))_t$ is predictable.
 
 \item[(B2)] There exists a positive constants $c_1$ such that for every $(t,i,j)\in[0,T]\times I\times I;\,i\neq j$, $w\in \Om,\, u\in U$ and  $\mu, \nu \in\mathcal{P}(I)$
 $$
\l_{ij}(t,w,\mu,u)\ge c_1>0.
 $$
 \item[(B3)] For $p=1,2$ and for every $t\in[0,T]$, $w\in \Om,\, u\in U$ and  $\mu\in \mathcal{P}_2(I)$, 
 $$
 \underset{i,j: \, j\neq i}\sum |j-i|^p\l_{ij}(t,w,\mu,u)\le C(1+|w|^p_t+\int|y|^p\mu(dy)).
 $$
 
 \item[(B4)]  For $p=1,2$ and for every $t\in[0,T]$, $w, \tilde w\in \Om$ and  $\mu, \nu \in\mathcal{P}(I)$,
  $$
  \underset{i,j: \, j\neq i}\sum |j-i|^p|\l_{ij}(t,w,\mu,u)-\l_{ij}(t,\tilde w,\nu,v)| \le C(|w-\tilde w|^p_t+d^p(\mu,\nu)+\delta^p(u,v)).
  $$
  \item[(B5)] For every $t\in[0,T]$, $w\in \Om,\, u\in U$ and  $\mu\in \mathcal{P}_1(I)$,
  $$
 \underset{i,j: \, j\neq i}\sum \l^2_{ij}(t,w,\mu)\le C(1+|w|_t+\int |y|\mu(dy)).
 $$
 \item [(B6)] There exists a constant $\a>0$ such that $\int e^{\a y}\xi(dy)<+\infty$. This condition implies that the probability measure $\xi$ on $I$ has finite second moment: $\|\xi\|_2^2:=\int |y|^2\xi(dy)<\infty$.
 
\end{itemize}

\ms\no Existence of $P^u$ such that $P^u\circ x^{-1}(0)=\xi$ is derived as a fixed point of $\Phi^u$ defined in the same way as in Theorem \ref{FP-tv} except that the intensities  $\l_{ij}(\cdot)$ further depend on $u$, which does not rise any major issues.  

\ms\no Let $P$  be the probability measure on $(\Omega, \mathcal F)$ under which $x$ is a time-homogeneous Markov chain such that $P\circ x^{-1}(0)=\xi$ and with $Q$-matrix $(g_{ij})_{ij}$  satisfying \eqref{G}, \eqref{exp-1} and \eqref{G2}.
We have
 \begin{equation}\label{P-u}
 dP^u:=L^u(T) dP,
 \end{equation}
where, for $0\le t\le T$,
\begin{equation}\label{P-u-density}
 L^u(t):=\underset{\substack{i,j\\ i\neq j} }\prod \exp{\left\{ \int_{(0,t]}\ln{\frac{ \l_{ij}^u(s)}{g_{ij}}}dN_{ij}(s)-\int_0^t( \l_{ij}^u(s)-g_{ij})I_i(s)ds  \right\}}, 
\end{equation}
 which satisfies
\begin{equation}\label{L-u-1}
L^u(t)=1+\int_{(0,t]} L^u(s^-)\underset{i,j:\, i\neq j}\sum I_i(s^-)\ell^u_{ij}(s)dM_{ij}(s),
\end{equation}
where $\ell^u_{ij}(s):=\ell_{ij}(t, x,P^u\circ x^{-1}(s),u(s))$ is given by the formula
\begin{equation}\label{L-u-2}
\ell^u_{ij}(s)=\left\{\begin{array}{rl}
\l_{ij}^u(s)/g_{ij}-1 &\text{if }\,\, i\neq j,\\ 0 & \text{if }\,\, i=j,
\end{array}
\right.
\end{equation}
and $(M_{ij})_{ij}$ is the $P$-martingale given in \eqref{mart-1}. Moreover, in a similar way as in \eqref{m-m-Q}, the accompanying martingale $M^u=(M^u_{ij})_{ij}$ satisfies
\begin{equation}\label{m-u-1}
M^u_{ij}(t)=M_{ij}(t)-\int_{(0,t]} \ell^u_{ij}(s)I_i(s^-)g_{ij}ds.
\end{equation}
The conditions (B5) and (B6) correspond to (A6) and (A7) of Theorem \eqref{FP-tv} and are imposed to guarantee that $L^u$ is $L^2(P)$-bounded.

\ms\no We first derive continuity of the map $u\mapsto P^u$ and then state the optimal control problem  we want to solve.

\ms\no
Let $E^u$  denote the expectation w.r.t. $P^u$. By \eqref{x-Q-bar}, we have, for every $u\in\U$,
\begin{equation}\label{u-x-estim}
\|P^u\|_2^2=E^u[|x|^2_T]\le Ce^{CT}(1+\|\xi\|_2^2)<\infty.
\end{equation}
 We further have the following estimate of the total variation between $P^u$ and $P^v$.
\begin{lemma}\label{TV-u}
For every $u,v\in \U$, it holds that
\begin{equation}\label{TV-uv-1}
D_T^2(P^u,P^v)\le C\sup_{0\le t\le T}E^u[\delta^2(u(t),v(t))].
\end{equation}
In particular, the function $u\mapsto P^u$ from $U$ into $\mathcal{P}_2(\Om)$ is Lipschitz continuous: for every $u,v\in U$, 
\begin{equation}\label{TV-uv-2}
D_T(P^u,P^v)\le C\delta(u,v).
\end{equation} 
Moreover, 
\begin{equation}\label{K}
K_T:=\sup_{u\in U}\|P^u\|_2\le C<\infty,
\end{equation}
for some constant $C>0$ that depends only on $T$ and $\xi$.
\end{lemma} 
\begin{proof} A similar estimate as \eqref{entropy-estimate} yields
\begin{equation}\label{entropy-estimate-u}
D^2_T(P^u,P^v)\le \frac{1}{c_1}\underset{i,j:\, i\neq j}\sum \,E^u\left[\int_0^T (\l^{u}_{ij}(t)-\l^{v}_{ij}(t))^2 \,dt\right].
\end{equation}
Using (B3), we obtain 
\begin{equation*}
D^2_T(P^u,P^v)\le \frac{C}{c_1}E^u\left[\int_0^T d^2(P^u(t),P^v(t))+\delta^2(u(t),v(t))dt\right]
\end{equation*}
By \eqref{margine} and Gronwall inequality we finally obtain
\begin{equation*}
D^2_T(P^u,P^v)\le C \sup_{0\le t\le T}E^u[\delta^2(u(t),v(t))].
\end{equation*}

\ms\no Inequality (\ref{TV-uv-2}) follows from (\ref{TV-uv-1}) by letting  $u(t)\equiv u\in U$ and $v(t)\equiv v\in U$. It remains to show (\ref{K}). But, this follows from (\ref{u-x-estim}) and the continuity of the function $u\mapsto P^u$ from the compact set $U$ into $\mathcal{P}_2(\Om)$. 
\end{proof}

In the rest of the paper, we will assume that $x(0)$ is a given deterministic point in $I$ and $\F_0$ is the trivial $\s$-algebra.  
 
 Let $f$ be a measurable functions from $[0,T]\times\Om\times\mathcal{P}_2(I)\times U$ into $\R$ and $h$ be a measurable functions from $I\times\mathcal{P}_2(I)$ into $\R$ such that \\

\begin{itemize}
 \item[(B7)] For any $u\in\U$ and $Q\in \mathcal{P}_2(\Om)$, the process  $(f(t, x,Q\circ x^{-1}(t),u(t)))_t$ is progressively measurable. Moreover, $h(x(T),Q\circ x^{-1}(T))$ is $\mathcal{F}_T$-measurable.
 \item[(B8)] For every $t\in[0,T]$, $w\in\Om$, $u,v\in U$ and $\mu, \nu \in\mathcal{P}_2(I)$,
  $$
  |\phi(t,w,\mu, u)-\phi(t,w,\nu,v)|\le C(d(\mu,\nu)+\delta(u,v)).
  $$
  for $\phi\in\{f,h\}$. 
  
 \item[(B9)] $f$ and $h$ are uniformly bounded. 
 \end{itemize}
  
\ms\no The cost functional $J(u),\,\, u\in\U$ associated with the controlled the jump process through the intensities $\l_{ij}(t,u(t))$ is 
\begin{equation}\label{J-u}
J(u):=E^u\left[\int_0^T f(t,x,P^u\circ x^{-1}(t),u(t))dt+ h(x(T),P^u\circ x^{-1}(T))\right],
\end{equation}
where $f$ and $h$ satisfy (B7), (B8) and (B9) above.  \\

\noindent Any $\widehat u\in\U$ satisfying
\begin{equation}\label{opt-J}
J(\widehat u)=\min_{u\in\U}J(u)
\end{equation}
is called optimal control. The corresponding optimal dynamics is given by the probability measure $\widehat P$ on $(\Om,\F)$ defined by
\begin{equation}\label{opt-P}
d\widehat P=L^{\widehat u}(T)dP,
\end{equation}
where $L^{\widehat u}$ is given by the same expression as \eqref{P-u-density} and under which the coordinate process $x$ is a jump process with intensities 
$$
\l_{ij}^{\widehat u}(t):=\l_{ij}(t,x,P^{\widehat u}\circ x^{-1}(t),\widehat u(t)),\,\,\, i,j\in I,\,\, 0\le t\le T.
$$
\\
We want to prove existence of such an optimal control and characterize the optimal cost functional $J(\widehat{u})$.\\

\noindent For $(t,w,\mu,u)\in [0,T]\times\Om\times \mathcal{P}_2(I) \times U$ and a matrix $p=(p_{ij})$ on $I\times I$ with real-valued entries,  we introduce the Hamiltonian associated with the optimal control problem (\ref{J-u})
\begin{equation}\label{ham-u}
H(t,w,\mu,p,u):=f(t,w,\mu,u)+ \langle \ell (t,w,\mu,u),p\rangle_g,
\end{equation}
where
$$
\langle\ell (t,w,\mu,u),p\rangle_g:=\underset{i,j:\, i\neq j}\sum p_{ij}\ell_{ij}(t, w,\mu,u)g_{ij}\mathbf{1}_{\{w(t^-)=i\}},
$$
Recalling that $\ell_{ij}(t, w,\mu,u)g_{ij}=\l_{ij}(t, w,\mu,u)-g_{ij}$ for $i\neq j$, we have 
\begin{equation*}
\langle\ell (t,w,\mu,u)-\ell (t,w,\nu,v),p\rangle_g=\underset{i,j:\, i\neq j}\sum p_{ij}(\l_{ij}(t,w,\mu,u)-\l_{ij}(t,w,\nu,v))\mathbf{1}_{\{w(t^-)=i\}}.
\end{equation*}
Using (B3) we obtain 
$$
|\l_{ij}(t, w,\mu,u)-\l_{ij}(t, w,\nu,v)|\le C(d(\mu,\nu)+\delta(u,v)),\quad j\neq i.
$$
Therefore, provided that $\|p\|_g(t)<+\infty$ (see the notation \eqref{g-t}), 
\begin{equation*}\begin{array}{lll}
|\langle \ell (t,w,\mu,u)-\ell (t,w,\nu,v),p\rangle_g|\le \frac{1}{c_2}\|p\|_g(t)\|\l (t,w,\mu,u)-\l(t,w,\nu,v)\|_g \\ \qquad\qquad\qquad\qquad\qquad\qquad\qquad \le C\|p\|_g(t)(d(\mu,\nu)+\delta(u,v))(\underset{i,j:\, i\neq j}\sum g_{ij})^{1/2}.
\end{array}
\end{equation*}
where $\underset{i,j:\, i\neq j}\sum g_{ij}<+\infty$, by \eqref{G}.\\
Hence, in view of (B3), the Hamiltonian $H$ satisfies
\begin{equation}\label{H-u-lip}
|H(t,w,\mu,p,u)-H(t,w,\nu,p,v)|\le C(1+\|p\|_g(t))(d(\mu,\nu)+\delta(u,v)).
\end{equation}
Moreover, noting that for $i\neq j$, $|i-j|\ge 1$, we may use (B4) to obtain 
$$
\ell_{ij}(t, w, \mu,u)\le 1+\frac{\l_{ij}(t, w, \mu,u)}{g_{ij}}\le 1+\frac{C}{c_2}(1+|w|_t+\int |y|\mu(dy)),
$$
that is
$$
\ell_{ij}(t, w, \mu,u)\le C(1+|w|_t+\int |y|\mu(dy)), \quad i,j\in I,\,\ i\neq j.
$$
Therefore, 
\begin{equation}\label{H-u-p}
|H(t,w,\mu,p,u)-H(t,w,\mu,p^{\prime},u)|\le C\left(1+|w|_t+\int |y|\mu(dy)\right)\|p-p^{\prime}\|_g(t).
\end{equation}

Next, we show that the cost functional $J(u),\,u\in\U$, can be expressed by means of solutions of  a linear BSDE.
\begin{proposition}\label{u-bsde} For every $u\in\U$, the BSDE
\begin{eqnarray}\label{u-yz-bsde} \qquad\quad \left\{\begin{array}{ll}
-dY^u(t)=H(t,x,P^u\circ x^{-1}(t),Z^u(t),u(t)) dt-Z^u(t)dM(t),\\ \quad 
Y^u(T)=h(x(T),P^u\circ x^{-1}(T)),
\end{array}
\right.
\end{eqnarray} 
admits a solution $(Y,Z)$ which consists of an  $\Ff$-adapted process $Y$ which is right-continuous with left limits and a predictable process $Z$ which satisfy 
\begin{equation}\label{u-bsde-estim}
E\left[|Y^u|^2_T+\int_0^T\|Z^u(s)\|_{g}^2ds\right]<+\infty.
\end{equation}
This solution is unique up to indistinguishability for $Y$ and equality 
$dP\times g_{ij}I_i(s^-)ds$-almost everywhere for $Z$.

\ms\no Moreover, $Y^u_0=J(u)$.
\end{proposition}
\begin{proof} Since the Hamiltonian $H(t,w,\mu,p,u)$ is linear in $p$, by Theorem \eqref{bsde-th}, existence and uniqueness of solutions of the BSDE (\ref{u-yz-bsde}) satisfying (\ref{u-bsde-estim}) follows from \eqref{H-u-p}, the boundedness of $h(x(T),P^u\circ x^{-1}(T))$ and  the boundedness of $H(t,x,P^u\circ x^{-1}(t),0,u(t))$  which follows from  (B9). 

It remains to show that $Y^u_0=J(u)$. Indeed, in terms of the $(\Ff, P^u)$-martingale
\begin{equation*}
M^u_{ij}(t)=M_{ij}(t)-\int_{(0,t]} \ell_{ij}^u(s)I_i(s^-)g_{ij}ds.
\end{equation*}
the process $(Y^u,Z^u)$ satisfies, for $0\le t\le T$,
$$
Y^u(t)=h(x(T),P^u\circ x^{-1}(T))+\int_t^T f(s,x,P^u\circ x^{-1}(s),u(s)) ds-\int_t^T Z^u(s)dM^u(s).
$$
Therefore, $P^u$-a.s.,
$$
Y^u(t)=E^u\left[\int_t^T f(s,x,P^u\circ x^{-1}(s),u(s))ds+h(x(T),P^u\circ x^{-1}(T))\big|\F_t\right]
$$
In particular,
$$
Y^u_0=E^u\left[\int_0^T f(s,x,P^u\circ x^{-1}(s),u(s))dt+h(x(T),P^u\circ x^{-1}(T))\right]=J(u),
$$
since $\F_0$ is the trivial $\s$-algebra.  
\end{proof}

\subsection{Existence of an optimal control}
In the remaining part of this section we want to find $\widehat u\in\U$ such that $\widehat u=\arg\min_{u\in\U}J(u)$. A way to find such an optimal control is to proceed as in Proposition \ref{u-bsde} and introduce a linear BSDE whose solution $Y^*$ satisfies $Y^*_0=\inf_{u\in\U}J(u)$. Then, by comparison (cf. Proposition \ref{bsde-comp}), the problem can be reduced to  minimizing the corresponding Hamiltonian and the terminal value $h$ w.r.t. the control $u$.  

As explained in the introduction, since  the marginal law $P^u\circ x_t^{-1}$ of $x_t$ under $P^u$ depends on the whole path of $u$ over $[0,t]$ and not only on $u_t$,  we should minimize the Hamiltonian $H(t,x_.,P^u\circ x_t^{-1},z,u_t)$ w.r.t. the whole set $\U$ of admissible stochastic controls. Therefore, we should take the essential infimum of the Hamiltonian over $\U$, instead of the minimum over $U$. Therefore, for the associated BSDE to make sense, we should show that it exists and is progressively measurable. This is shown in the next proposition. 

 Let $\mathbb{L}$ denote the $\sigma$-algebra of progressively measurable sets on $[0,T]\times\Omega$. For  $z\in \R^{I\times I}$, the set of real-valued $I\times I$-matrix, set
\begin{equation}\label{def-H}
H(t,x_.,z,u):=H(t,x_.,P^u\circ x_t^{-1},z,u_t).
\end{equation}
Since $H$ is linear in $z$ and a progressively measurable process, it is an $\mathbb{L}\times B(\R^{I\times I})$-random variable.

\begin{proposition}\label{ess-inf}
There exists an $\mathbb{L}$-measurable process $H^*$ such that for every $z\in \R^{I\times I}$,
\begin{equation}\label{u-opt-1}
H^*(t,x_.,z)=\mathrm{ess}\inf_{u\in \U}H(t,x_.,z,u),\quad dP \times dt \mbox{-a.s.} 
\end{equation} 
Moreover, $H^*$ is Lipschitz continuous in $z$: For every $z,z^{\prime} \in\R^{I\times I}$,
\begin{equation}\label{H*-lip}
|H^*(t,x,z)-H^*(t,x,z^{\prime})|\le C(1+|x|_t+\sup_{u\in U}\|P^u\|_2)\|z-z^{\prime}\|_g(t).
\end{equation}

If $\theta$ is an $\mathbb{L}$-measurable process with values in $\R^{I\times I}$, then
\begin{equation}\label{u-opt-2}
H^*(t,x_.,\theta_t)=\underset{u\in\U}{\mathrm{ess}\inf\,}H(t,x_.,\theta_t,u),
\,\,dP \times
dt \mbox{-a.e.}
\end{equation}
\end{proposition}

\begin{proof}
The proof of \eqref{u-opt-1} and \eqref{u-opt-2} is similar to the one of Propositions 4.4 and 4.6 in \cite{DH}. We give it in the appendix  for the sake of completeness.  We prove the inequality \eqref{H*-lip}. We have
\begin{eqnarray*}\begin{array}{lll}
|H^*(t,x,z)-H^*(t,x,z^{\prime}) \\ \qquad\qquad\quad= \left|\underset{u\in \U}{\mathrm{ess}\inf\,}H(t,x,P^u\circ x^{-1}_t,z,u)-\underset{u\in\U}{\mathrm{ess}\inf\,}H(t,x,P^u\circ x^{-1}_t,z^{\prime},u)\right|\\ \qquad\qquad\quad\le \underset{u\in \U}{\mathrm{ess}\sup\,}\left|H(t,x,P^u\circ x^{-1}_t,z,u)-H(t,x,P^u\circ x^{-1}_t,z^{\prime},u)\right |\\ \qquad\qquad\quad\le C(1+|x|_t+ \sup_{u\in U}\|P^u\|_2)\|z-z^{\prime}\|_g(t),
\end{array}
\end{eqnarray*}
by \eqref{H-u-p}, where by the continuity of $u\mapsto P^u$,  $K_T:=\sup_{u\in U}\|P^u\|_2$ is finite. 
\end{proof}

\medskip Define the $\F_T$-measurable random variable
\begin{equation}\label{h-*}
h^*(x):=\underset{u\in\U}{\mathrm{ess}\inf\,} h(x(T),P^u\circ x^{-1}(T)),
\end{equation}
and let $(Y^*,Z^*)$ be the solution of the BSDE
\begin{equation}\label{opt-bsde}
\left\{\begin{array}{lll}
-dY^*(t)=H^*(t,x,Z^*(t))dt-Z^*(t)dM(t),\quad 0\le t<T,\\ Y^*_T=h^*(x),
\end{array}
\right.
\end{equation}

We have the following comparison result.
\begin{proposition}[Comparison result]\label{comp-bsde-u}
For every $t\in[0,T]$, it holds that
\begin{equation}
Y^*(t)\le Y^u(t),\quad P\mbox{-a.s.},\quad u\in\U.
\end{equation}
\end{proposition}

\begin{proof} We have
$$
\begin{array}{lll}
Y^*(t)-Y^u(t)=h^*(x_.)-h(x(T),P^u\circ x_T^{-1})-\int_t^T (Z^*(s)-Z^u(s))dM(s)\\ +\int_t^T \{H^*(s,x_.,Z^*(s))-H(s,x,P^u\circ x^{-1}(s),Z^*(s),u(s))\} ds\\ +\int_t^T \{H(s,x,P^u\circ x^{-1}(s),Z^*(s),u(s)) -H(s,x,P^u\circ x^{-1}(s),Z^u(s),u(s))\} ds. 
\end{array}
$$
using the definition of $H^*$ and $h^*$ and noting that
$$
\begin{array}{lll}
H(s,x,P^u\circ x^{-1}(s),Z^*(s),u(s)) -H(s,x,P^u\circ x^{-1}(s),Z^u(s),u(s))\\ \qquad\qquad\qquad\qquad\qquad\qquad\qquad =\langle \ell(s,u(s)),Z^*(s)-Z^u(s)\rangle_g,
\end{array}
$$
we have
$$
\begin{array}{lll}
Y^*(t)-Y^u(t)\ge \int_t^T\langle \ell(s,u(s)),Z^*(s)-Z^u(s)\rangle_g ds-\int_t^T (Z^*(s)-Z^u(s))dM(s)\\ \qquad\qquad\qquad =\int_t^T (Z^*(s)-Z^u(s))dM^u(s),
\end{array}
$$
where $M^u$ is the $P^u$- martingale defined in \eqref{m-u-1}. Taking the $P^u$-conditional expectation w.r.t. $\F_t$, we obtain  $Y^*(t)\le Y^u(t),\, \forall u\in\U$.
\end{proof}

\begin{proposition}[$\ep$-optimality] Assume that for any $\ep>0$ there exists $u^{\ep}\in \U$ such that $P$-a.s.,
\begin{equation}\label{e-*-H-g}
\left\{ \begin{array}{ll} H^*(t,x_.,Z^*(t))\ge H(t,x_.,Z^*(t),P^{u^{\ep}}\circ x^{-1}(t),u^{\ep}(t))+\ep, \quad 0\le t<T, \\ h^*(x_.)\ge h(x(T),P^{u^{\ep}}\circ x^{-1}(T))+\ep.
\end{array}
\right.
\end{equation}
Then, 
\begin{equation}\label{*-e-opt}
Y^*(t)=\underset{u\in\U}{\mathrm{ess}\inf\,} Y^u(t),\quad 0\le t\le T.
\end{equation}
\end{proposition}

\begin{proof} Let $(Y^{\ep},Z^{\ep})$ be the solution the following BSDE
$$
\begin{array}{lll}
Y^{\ep}(t)=h(x(T),P^{u^{\ep}}\circ x^{-1}(T))+\int_t^T H(s,x_.,Z^{\ep}(s),P^{u^{\varepsilon}}\circ x^{-1}(s),u^{\ep}(s))ds\\\qquad\qquad -\int_t^T Z^{\ep}(s)dM(s).
\end{array}
$$ 
We have
$$
\begin{array}{lll}
Y^*(t)-Y^{\ep}(t)=h^*(x_.)-h(x_T,P^{u^{\ep}}\circ x_T^{-1})-\int_t^T (Z^*(s)-Z^{\ep}(s))dM(s)\\ \,+\int_t^T \{H^*(s,x_.,Z^*)-H(t,x_.,P^{u^{\ep}}\circ x^{-1}(s),Z^*(s),u^{\ep}(s))\} ds\\ \,+\int_t^T \{H(t,x_.,P^{u^{\ep}}\circ x^{-1}(s),Z^*(s),u^{\ep}(s)) -H(s,x_.,P^{u^{\ep}}\circ x^{-1}(s),Z^{\ep}(s),u^{\ep}(s))\} ds. 
\end{array}
$$
Since 
$$
H^*(s,x_.,Z^*(s))-H(s,x_.,P^{u^{\ep}}\circ x^{-1}(s),Z^*(s),u^{\ep}(s))\ge \ep,
$$
and
$h^*(x_.)-h(x_T,P^{u^{\ep}}\circ x_T^{-1})\ge \ep $,  applying a similar argument as in the proof of Proposition \eqref{comp-bsde-u}, we obtain   $Y^*(t)\ge Y^{u^{\ep}}(t) +\ep(T+1)$. Therefore, for every $0\le t\le T$,  $Y^*(t)=\underset{u\in\U}{\mathrm{ess}\inf\,}Y^u(t)$ .
\end{proof}

In next theorem, we characterize the set of optimal controls associated with  (\ref{opt-J}) under the dynamics $P^u$.
\begin{theorem}[Existence of optimal control]\label{opt-BSDE} If there exists $u^*\in\U$ such that 
\begin{equation}\label{*-H-opt}
H^*(t,x,Z^*(t))=H(t,x,P^{u^*}\circ x^{-1}(t),Z^*(t),u^*(t)),\quad 0\le t< T,
\end{equation}
and
\begin{equation}\label{*-g-opt}
h^*(x)=h(x(T),P^{u^*}\circ x^{-1}(T)).
\end{equation}
Then, \begin{equation}\label{Y-opt}
Y^*(t)=Y^{u^*}(t)=\underset{u\in\U}{\mathrm{ess}\inf\,} Y^u(t),\quad 0\le t\le T.
\end{equation}
In particular, $Y_0^*=\inf_{u\in\U} J(u)=J(u^*)$.
\end{theorem}

\begin{proof} By comparison, the conditions \eqref{*-H-opt} and \eqref{*-g-opt} imply that $Y^*=Y^{u^*}$.  Due to \eqref{*-e-opt}, we arrive at \eqref{Y-opt}. 
\end{proof}

\begin{remark}\label{selection}
If the marginal law $P^u\circ x_s^{-1}$ of $x_s$ under $P^u$ is a function of $ (x,u(s))$ only and does not depend on the whole path of $u$ over $[0,s]$, it suffices to take the minimum of  $H$ and $h$  over the compact set of controls $U$, instead of taking the essential infimum over $\U$.  An optimal control over $[0,T]$ can be obtained by pasting the minima of $H$ and $h$ as follows.  By Bene\v{s} selection theorem \cite{Benes}, there exist  two measurable
functions $u_1^*$ from $[0,T)\times\Om\times\R^{I\times I}$ into $U$ and
$u_2^*$ from $I$ into $U$ such that
\begin{equation*}
H^*(t,x,z):=\inf_{u\in U}H(t,x,P^u\circ
x_t^{-1},z,u)= H(t,x,P^{u_1^*}\circ x_t^{-1},z,u_1^*(t,x,z))
\end{equation*}  and
\begin{equation*}
h^*(x):=\inf_{u\in U} h(x_T,P^u\circ
x_T^{-1})=h(x_T,P^{u_2^*}\circ x_T^{-1}).
\end{equation*}
Thus,  the progressively measurable function $u^*$ defined by
\begin{equation*}\label{u-opt-hat1}
\widehat u(t,x,z):=\left\{\begin{array}{ll} u_1^*(t,x,z), \quad
t<T,\\ u_2^*(x_T),\quad t=T,
\end{array}
\right.
\end{equation*}
satisfies
\begin{equation*}\label{opt}
H^*(t,x,z)=H(t,x,P^{\widehat u}\circ x_t^{-1},z,\widehat u)
\quad  \text{and} \quad   h^*(x)=h(x_T,P^{\widehat u}\circ
x_T^{-1}).\qed
\end{equation*}
\end{remark}

We end this section by providing an example where an optimal control exists. 
\begin{example}  Assume that  the set of $L^2$-bounded densities $\{ L^u_T, \,u\in \U\}$ is weakly sequentially compact for the topology $\sigma(L^1,L^{\infty})$.  Consider a cost functional of the form
\begin{equation*}
J(u)=E^u\left[\int_0^T f(t,x,E^u[\a(x(t))],u(t))dt+ h(x(T),E^u[\beta (x(T))])\right],
\end{equation*} 
where $\a, \beta, f, h$ are bounded functions and $(y,a)\in \R\times U\mapsto f(\cdot,\cdot,y,a)$ and $y\in\R \mapsto h(\cdot,y)$  are continuous. Then, an optimal control exists.  Indeed,  let $(u_n)_{n\geq 0}$ be a sequence in $\U$ such that 
$$
\inf_{u\in \U}J(u)=\lim_{n\rightarrow \infty}J(u_n). 
$$
By weak compactness of the set of densities $\{L^{u}_T, u\in \U\}$,  there exist $u^*\in \U$ and a subsequence $(L^{u_{n_k}}_T)_{ k\geq 0}$ which converges weakly to 
$L^{u^*}_T$. Since $\a, \beta$ are bounded we have, for any $t\leq T$,
$$
\begin{array}{lll}
\underset{k\to \infty}\lim E^{u_{n_k}}[\a(x(t))]= \underset{k\to \infty}\lim E[L^{u_{n_k}}_T\a(x(t))]=E[L^{u^*}_T\a(x(t))]=E^{u^*}[\a(x(t))] \\ \underset{k\to \infty}\lim E^{u_{n_k}}[\beta (x(T))]=\underset{k\to \infty}\lim  E[L^{u_{n_k}}_T\b(x(T))]= E[L_T^{u^*}\b(x(T))]=E^{u^*}[\b(x(T))].
\end{array}
$$
Using the boundedness and continuity of $f$ and  $h$, by the dominated convergence theorem, we obtain
\begin{equation*}\begin{array}{lll}
\underset{k\to \infty}\lim\int_0^Tf(s,x,E^{u^{n_k}}[\a(x(t))])ds+h(x(T),
E^{u_{n_k}}[\b(x(T))])\\ \qquad\qquad\qquad=\int_0^T f(s,x,E^{u^{*}}[\a(x(t))])ds+h(x(T),
E^{u{*}}[\b(x(T))])\end{array}
\end{equation*} in $L^p$, for any $p\geq 1$. Thus,  the weak convergence of $(L_T^{u_{n_k}})_{k\geq 0}$, we have that 
$\underset{k\to \infty}\lim J(u_{n_k})=J(u^*)$ which implies that $J(u^*)=\inf_{u\in \U}J(u),
$ i.e. $u^*$ is optimal for $J$. 

If the set of intensities $\lambda^u$ is mean-field free i.e. $\lambda^u(t):=\lambda(t,x,u(t))$ and for each $(t,w)\in [0,T]\times \Omega$,  $\lambda (t,y,U)$ is convex (Roxin's condition) i.e. for every $u_1,u_2\in\U$, and $a\in [0,1]$, there exists an admissible control $u\in \U$ such that
$$
\lambda(t,w,u(t,w))=a\lambda(t,w,u_1(t,w))+(1-a)\lambda(t,w,u_2(t,w)),
$$
then, mimicking the proof of Theorems 3 and 4 in \cite{Benes}, the set of densities $\{L^{u}_T, u\in \U\}$ is convex and weakly sequentially closed. Being $L^2$-bounded, it is weakly  sequentially compact. The proof of these results relies on the measurable selection theorem, which does not seem extend to intensities $\lambda^u$ of mean-field type which, at each time $t$, depend on the whole path of $u$ over $[0,t]$. 

\end{example}

\subsection{Existence of nearly-optimal controls}
The sufficient condition \eqref{*-H-opt}-\eqref{*-g-opt} is quite hard to verify in concrete situations. This makes Theorem
\eqref{opt-BSDE} less useful for showing existence of optimal controls. Nevertheless, near-optimal controls enjoy many useful and desirable properties that optimal controls do not have. Thanks to Ekeland's variational principle \cite{Ekeland}, that we will use below, under very mild conditions on the control set $\U$ and the payoff functional $J$, near-optimal controls always exist while optimal controls may not exist or are difficult to establish. 

We introduce the Ekeland metric $d_E$ on the space
$\U$ of admissible controls defined as follows. For $u, v\in \U$,
\begin{equation}\label{E-distance}
d_E(u,v):=\widehat P\{(\omega,t)\in\Omega\times[0,T],\,\,
\delta(u_t(\omega),v_t(\omega))>0\},
\end{equation}
where $\widehat P$ is the product measure of $P$ and the Lebesgue
measure on $[0,T]$.

In our proof of existence of near-optimal controls, we use $L^2$-boundedness of the Girsanov density $L^u$ , which, in view to Theorem \eqref{FP-tv}, is satisfied under assumptions (B6) and (B7). We have

\begin{lemma}\label{ekeland}
\begin{itemize}
\item[$(i)$] $d_E$ is a distance. Moreover, $(\U,d_E)$  is a complete metric space.
 \item[$(ii)$]  Let $(u^n)_n$ and $u$  be in $\U$. If $d_E(u^n,u)\to 0$ then $\E[\int_0^T\delta^2(u^n_t,u_t)dt]\to 0$.
\end{itemize}
\end{lemma}

\begin{proof}
For a proof of $(i)$, see \cite{Elliott}. The proof of completeness of $(\U,d_E)$ needs only completeness of the metric space $(U,\delta)$.\\
$(ii)$ Let $(u^n)_n$ and $u$ be in $\U$. Then, by definition of the distance $d_E$,  since $d_E(u^n,u)\to 0$ then $\delta(u^n_t,u_t)$ converges to 0, $dP\times dt$-a.e. Now, since the set $U$ is compact, the sequence $\delta(u^n,u)$ is  bounded. Thus, by dominated convergence, we have $\E[\int_0^T\delta^2(u^n_t,u_t)dt]\to 0$.
\end{proof}

\begin{proposition}\label{D-conv}
Assume (B1) to (B9) hold and let $(u^n)_n$ and $u$ be in $\U$. If $d_E(u^n,u)\to 0$ then $D^2_T(P^{u^n},P^u)\to 0$.
Moreover, for every $t\in[0,T]$, $L^{u^n}_t$ converges to $L^{u}_t$ in $L^1(P)$.
\end{proposition}
\begin{proof}
In view of  Lemma \eqref{ekeland}, we have $\E[\int_0^T\delta^2(u_t, u^n_t)dt] \to 0$. Therefore, the sequence $(\int_0^T
\delta^2(u_t, u^n_t)dt)_{n}$ converges in probability w.r.t $P$ to 0 and by compactness of $U$, it is uniformly bounded. On the other hand, since $L^u_T$ is integrable then the sequence $(L^u_T\int_0^T \delta^2(u_t, u^n_t)dt)_{n}$ converges also in probability (w.r.t. to $P$) to 0. Next,  by the uniform boundedness of $(\int_0^T\delta^2(u_t, u^n_t)dt)_{n}$, the sequence $(L^u_T\int_0^T\delta^2(u_t, u^n_t)dt)_{n}$ is uniformly integrable. Finally, since 
$$
\E^u[\int_0^T \delta^2(u_t, u^n_t)dt]=\E[L^u_T\int_0^T \delta^2(u_t,
u^n_t)dt],
$$ 
it follows that  $\E^u[\int_0^T \delta^2(u_t, u^n_t)dt]\rightarrow 0$ as $n\to +\infty$. To conclude it is enough to use the inequality \eqref{TV-uv-1}.

To prove that $L^{u^n}_t$ converges to $L^u_t$  in $L^1(P)$,  using the fact that  $(L^{u^n}_t)_n$ is $L^2(P)$-bounded (hence, uniformly integrable) it suffices to show that $L^{u^n}_t$ converges to $L^u_t$ in probability w.r.t. $P$, as $n\to +\infty$.  Using the following relationship between the Hellinger distance $\widehat{D}_t(P^{u^n},P^u)$
and the Total variation distance $D_t(P^{u^n},P^u)$:
$$
E\left[\left(\sqrt{L^{u^n}_t}-\sqrt{L^{u}_t}\right)^2\right]:=\widehat D_t^2(P^{u^n},P^u)\le 2D_t(P^{u^n},P^u),
$$
and  that $D_t(P^{u^n},P^u)$ tends to $0$, we obtain $\sqrt{L^{u^n}_t}$ converges to $\sqrt{L^{u}_t}$ in probability (P), as $n\to +\infty$. This  in turn yields that $L^{u^n}_t$ converges to $L^u_t$ in probability (P), as $n\to +\infty$.
\end{proof}
\begin{proposition}[Existence of near optimal control]\label{near-opt}
For any $\ep>0$, there exists a control $u^{\ep}\in\U$ such that

\begin{equation}\label{e-opt}
J(u^{\ep})\le \inf_{u\in\U} J(u)+\ep.
\end{equation}
$u^{\ep}$ is called  near or $\ep$-optimal for the payoff functional $J$.
\end{proposition}
\begin{proof} The result follows from Ekeland's variational principle, provided that we prove that the payoff function $J$, as a mapping from the complete metric space $(\U,d_E)$ to $\R$, is lower bounded and lower-semicontinuous. Since $f$ and $h$ are assumed uniformly bounded, $J$ is obviously bounded. We now show continuity of $J$: $J(u^n)$ converges to $J(u)$ when  $d_E(u^n,u)\to 0$.

Integrating by parts, we obtain
$$
J(u)=E[\int_0^T L^u_tf(t,x,P^u\circ x_t^{-1},u_t)dt+L^u_Th(x_T,
P^u\circ x_T^{-1})].
$$
Using the inequality
$$
|L^{u^n}_tf(t,x,u^n)-L^{u}_tf(t,x,u)|\le
|L^{u^n}_t-L^{u}_t|f(t,x_.,u)|+L^{u}_t|f(t,x,u^n)-f(t,x,u)|,
$$
where, we set  $f(t,x,u):=f(t,x,P^u\circ x_t^{-1},u_t)$, and (B8) together with the boundedness of $f$, by Proposition \eqref{D-conv}, $E[\int_0^T L^{u^n}_tf(t,x,u^n)dt]$ converges to $E[\int_0^T L^u_tf(t,x,u)dt],$ as $d_E(u^n,u)\to 0$. A similar argument yields convergence of $E[L^{u^n}_Th(x_T, P^{u^n}\circ x_T^{-1})]$ to $E[L^u_Th(x_T, P^u\circ x_T^{-1})]$ when $d_E(u^n,u)\to 0$.
\end{proof}

%%%%%%%%%%%%%%Zero-game part%%%%%%%%%%%%%%%%
%%%%%%%%%%%%%%%%%%%%%%%%%%%%%

\section{The two-players zero-sum game problem}\label{zero-sum}
In this section we consider a two-players zero-sum game.  Let $\U$ (resp. $\V$) be the set of admissible $U$-valued (resp. $V$-valued) control strategies for the first (resp. second) player, where $(U,\delta_1)$ and $(V,\delta_2)$ are compact metric spaces.\\

For $(u,v),(\bar u,\bar v)\in U\times V$, we set
\begin{equation}\label{delta-u-v}
\delta((u,v),(\bar u,\bar v)):=\delta_1(u,\bar u)+\delta_2(v,\bar v).
\end{equation}
The distance $\delta$ defines a metric on the compact space $U\times V$. 

Let $P$  be the probability measure on $(\Omega, \mathcal F)$ under which $x$ is a time-homogeneous Markov chain such that $P\circ x^{-1}(0)=\xi$ and with $Q$-matrix $(g_{ij})_{ij}$  satisfying \eqref{G}, \eqref{exp-1} and \eqref{G2}.

 For $(u,v)\in\U\times \V$, let $P^{u,v}$ be the measure on $(\Om,\F)$ defined by
 \begin{equation}\label{P-u-v}
 dP^{u,v}:=L_T^{u,v}dP,
 \end{equation}
where
\begin{equation}\label{P-u-v-density}
 L^{u,v}(t):=\underset{\substack{i,j\\ i\neq j} }\prod \exp{\left\{ \int_{(0,t]}\ln{\frac{ \l^{u,v}_{ij}(s)}{g_{ij}}}dN_{ij}(s)-\int_0^t( \l^{u,v}_{ij}(s)-g_{ij})I_i(s)ds  \right\}}, 
\end{equation}

\begin{equation}\label{u-v-lambda}
\l^{u,v}_{ij}(t):=\l_{ij}(t,x,P^u\circ x^{-1}(t),u(t),v(t)),\,\,\, i,j\in I,\,\, 0\le t\le T,
\end{equation}
satisfying the following assumptions. 
\begin{itemize}
 \item[(C1)] For any $(u,v)\in\U\times \V$, $i,j\in I$, the process $((\l_{ij}(t, x,P^u\circ x^{-1}(t), u(t),v(t)))_t$ is predictable.
 
 \item[(C2)] There exists a positive constants $c_1$ such that for every $(t,i,j)\in[0,T]\times I\times I;\,i\neq j$, $w\in \Om,\, u\in U, v\in V$ and  $\mu, \nu \in\mathcal{P}(I)$
 $$
\l_{ij}(t,w,\mu,u,v)\ge c_1>0.
 $$
 
\item[(C3)] For $p=1,2$ and for every $t\in[0,T]$, $w\in \Om,\, u\in U, v\in V$ and  $\mu\in \mathcal{P}_2(I)$, 
 $$
 \underset{i,j: \, j\neq i}\sum |j-i|^p\l_{ij}(t,w,\mu,u,v)\le C(1+|w|^p_t+\int|y|^p\mu(dy)).
 $$
  \item[(C4)]  For $p=1,2$ and for every $t\in[0,T]$, $w, \tilde w\in \Om, (u,v), (\tilde u,\tilde v)\in U\times V$ and  $\mu, \nu \in\mathcal{P}(I)$,
  $$
  \begin{array}{lll}
  \underset{i,j: \, j\neq i}\sum |j-i|^p|\l_{ij}(t,w,\mu,u,v)-\l_{ij}(t,\tilde w,\nu,\tilde u,\tilde v)| \le C(|w-\tilde w|^p_t+d^p(\mu,\nu)\\ \qquad\qquad\qquad\qquad\qquad\qquad+\delta^p((u,v),(\tilde u,\tilde v)).
  \end{array}
  $$
  \item[(C5)] For every $t\in[0,T]$, $w\in \Om,\, u\in U, v\in V$ and  $\mu\in \mathcal{P}_1(I)$,
  $$
 \underset{i,j: \, j\neq i}\sum \l^2_{ij}(t,w,\mu, u,v)\le C(1+|w|_t+\int |y|\mu(dy)).
 $$
\item [(C6)] There exists a constant $\a>0$ such that $\int e^{\a y}\xi(dy)<+\infty$. 
 
 \end{itemize}

\ms\no By Proposition \eqref{L-Q-mart}, these assumptions guarantee that  $P^{u,v}$ is a probability measure on  $(\Omega,\F)$ under which the coordinate process $x$ is a chain with intensity matrix $\l^{u.v}$. Let $E^{u,v}$  denote the expectation w.r.t. $P^{u,v}$.\\

\ms\no Let $f$  be a measurable function from $[0,T]\times\Om\times\mathcal{P}_2(I)\times U\times V$ into $\R$ and $h$ be a measurable function from $I\times\mathcal{P}_2(I)$ into $\R$ such that \\

\begin{itemize}
 \item[(C7)] For every $t\in[0,T]$, $w\in\Om$, $(u,v),(\bar u,\bar v) \in U\times V$ and $\mu, \nu \in\mathcal{P}_2(I)$,
  $$
  |\phi(t,w,\mu, u,v)-\phi(t,w,\nu,\bar u,\bar v)|\le C(d(\mu,\nu)+\delta((u,v),(\bar u,\bar v)),
  $$
  for $\phi\in\{f,h\}$.
 \item[(C8)] For every $t\in[0,T]$, $w\in\Om$, $(u,v)\in U\times V$ and $\mu\in \mathcal{P}_2(I)$, 
 $$
 |f(t,w,\mu,u,v)|\le C(1+|w|_t+\int|y|\mu(dy)).
 $$
 \item[(C9)] $f$ and $h$ are uniformly  bounded.   
 \end{itemize}

\noindent The performance functional $J(u,v),\,(u,v)\in\U\times\V$, associated with the controlled Markov chain is 
\begin{equation}\label{J-u-v}
J(u,v):=E^{u,v}\left[\int_0^T f(t,x,P^{u,v}\circ x^{-1}(t),u(t),v(t))dt+ h(x(T),P^{u,v}\circ x^{-1}(T))\right].
\end{equation}
  \\
The zero-sum game we consider is between two players, where the first player (with control $u$) wants to minimize the payoff (\ref{J-u-v}), while  the second player (with control $v$) wants to maximize it.  The zero-sum game boils down to showing existence of a saddle-point for the game i.e.  to show existence of a pair $(\widehat u, \widehat v)$ of strategies such that  
\begin{equation}\label{J-u-v-hat}
J(\hat u, v) \le J(\widehat u, \widehat v)\le J(u,\widehat v)
\end{equation}
for each $(u, v)\in\U\times\V$.\\
 The corresponding optimal dynamics is given by the probability measure $\widehat P$ on $(\Om,\F)$ defined by
\begin{equation}\label{opt-P}
d\widehat P=L_T^{\widehat u, \widehat v}dP
\end{equation}
under which the chain has intensity $\l^{\widehat u, \widehat v}$.

\ms\noindent For $(t,w,\mu,u)\in [0,T]\times\Om\times\mathcal{P}_2(I)\times U\times V$ and  matrices $p=(p_{ij})$ with real-valued entries, we introduce the Hamiltonian associated with the optimal control problem (\ref{J-u-v})
\begin{equation}\label{ham-u-v}
H(t,w,\mu,p,u,v):=f(t,w,\mu,u,v)+ \langle \ell(t, w,\mu,u,v),p\rangle_g,
\end{equation}
where we recall that $\ell_{ij}(t, w,\mu,u,v)g_{ij}=\l_{ij}(t, w,\mu,u,v)-g_{ij}$ for $i\neq j$.

\ms\no In a similar way as for \eqref{H-u-lip} and \eqref{H-u-p}, whenever $\|p\|_g(t)$ and $\|p^{\prime}\|_g(t)$ are finite, the Hamiltonian $H$ satisfies
\begin{equation}\label{H-u-v-lip}
|H(t,w,\mu,p,u,v)-H(t,w,\nu,p,\bar u,\bar v)|\le C(1+\|p\|_g(t))(d(\mu,\nu)+\delta((u,\bar u),(v,\bar v)),
\end{equation}
and 
\begin{equation}\label{H-u-v-p}
|H(t,w,\mu,p,u,v)-H(t,w,\mu,p^{\prime},u,v)| \le C(1+|w|_t+\int |y|\mu(dy))\|p-p^{\prime}\|_g(t).
\end{equation}

Next, let $z\in R^{I\times I}$ and set 
\begin{itemize}
\item $\underline{H}(t,x,z):=\underset{v\in\V}{\mathrm{ess}\sup}\, \underset{u\in\U}{\mathrm{ess}\inf}\, H(t,x,z,u,v),$
\item $\overline{H}(t,x,z):=\underset{u\in\U}{\mathrm{ess}\inf}\, \underset{v\in\V}{\mathrm{ess}\sup}\, H(t,x,z,u,v),$
\item $\underline{h}(x):=\underset{v\in\V}{\mathrm{ess}\sup}\, \underset{u\in\U}{\mathrm{ess}\inf}\, h(x(T), P^{u,v}\circ x^{-1}(T))$,
\item $\overline{h}(x):=\underset{u\in\U}{\mathrm{ess}\inf}\, \underset{v\in\V}{\mathrm{ess}\sup}\, h(x(T), P^{u,v}\circ x^{-1}(T))$,
\item $(\underline{Y},\underline{Z})$ the solution of the BSDE associated with $(\underline{H}, \underline{h})$ and  $(\overline{Y},\overline{Z})$ the solution of the BSDE associated with $(\overline{H}, \overline{h})$.
\end{itemize}

Following a similar proof as the one leading to \eqref{H*-lip}, $\underline{H}(t,x_.,p)$ and $\overline{H}(t,x_.,p)$ are Lipschitz continuous in $p$ with the Lipschitz constant $C(1+|x|_t+\sup_{U\times V}\|P^{u,v}\|_2)$.

\begin{definition}[Isaacs' condition]
We say that the Isaacs' condition holds for the game if
$$
\left\{\begin{array}{lll}
\underline{H}(t,x,z)=\overline{H}(t,x,z),\quad 0\le t\le T, \\
\underline{h}(x)=\overline{h}(x)
\end{array}
\right.
$$
\end{definition}

Applying the comparison theorem for BSDEs, we obtain the following

\begin{proposition}\label{game-comparison} For every $t\in[0,T]$, it holds that $\underline{Y}_t\le \overline{Y}_t$, $\,P$-a.s.. Moreover, if the Issac's condition holds, then  
\begin{equation}\label{nash}
\underline{Y}(t)=\overline{Y}(t):=Y(t),\quad P\text{-a.s.},\quad 0\le t\le T.
\end{equation}
\end{proposition}  

In the next theorem, we formulate conditions for which the zero-sum game has a value.  For $(u,v)\in\U\times \V$, let $(Y^{u,v},Z^{u,v})$ be the solution of the BSDE 
\begin{equation}\label{u-v-yz-bsde}\left\{\begin{array}{lll}
-Y^{u,v}(t)=H(t,x,P^{u,v}\circ x^{-1}(t),Z^{u,v}(t),u(t),v(t)) dt-Z^{u,v}(t)dM(t),\,\, 0\le t<T,\\
Y^{u,v}(T)=g(x(T),P^{u,v}\circ x^{-1}(T)),
\end{array}
\right.
\end{equation} 

\begin{theorem}[Existence of a value of the game]\label{value-game}
Assume that, for every $0\le t<T$, 
$$
\underline{H}(t,x,Z(t))=\overline{H}(t,x,Z(t)).
$$
If there exists $(\widehat{u},\widehat{v})\in\U\times\V$ such that, for every $0\le t<T$, 
\begin{equation}\label{sp-H}
\underline{H}(t,x,Z(t))=\underset{u\in \U}{\mathrm{ess}\inf}\, H(t,x,Z(t),u,\widehat{v})=\underset{v\in \V}{\mathrm{ess}\sup}\, H(t,x,Z(t),\widehat{u},v),
\end{equation}
and
\begin{equation}\label{sp-g}
\underline{h}(x)=\overline{h}(x)=\underset{u\in \U}{\mathrm{ess}\inf}\, h(x(T),P^{u,\widehat{v}}\circ x^{-1}(T))=\underset{v\in \V}{\mathrm{ess}\sup}\, \, h(x(T),P^{\widehat{u},v}\circ x^{-1}(T)).
\end{equation}
Then, 
\begin{equation}\label{value}
Y(t)=\underset{u\in \U}{\mathrm{ess}\inf}\,\underset{v\in \V}{\mathrm{ess}\sup}Y^{u,v}(t)=\underset{v\in \V}{\mathrm{ess}\sup}\,\underset{u\in \U}{\mathrm{ess}\inf}\,Y^{u,v}(t), \quad 0\le t\le T.
\end{equation}
Moreover, the pair $(\widehat{u},\widehat{v})$ is a saddle-point for the game.
\end{theorem}

\begin{proof} Let $(u,v)\in\U\times \V$ and $(\widehat Y^u,\widehat Z^u)$ and $(\widetilde Y^v,\widetilde Z^v)$ be the solution of the following BSDE 
\begin{equation}\label{u-v-yz-bsde}\left\{\begin{array}{ll}
-\widehat Y^{u}(t)=\underset{v\in \V}{\mathrm{ess}\sup}\,H(t,x,\widehat Z^{u}(t),u,v) dt-\widehat Z^{u}(t)dM(t),\quad 0\le t<T,\\
\widehat Y^{u}(T)=\underset{v\in \V}{\mathrm{ess}\sup}\,h(x(T),P^{u,v}\circ x^{-1}(T)),
\end{array}
\right.
\end{equation} 
\begin{equation}\label{u-v-yz-bsde}\left\{\begin{array}{ll}
-\widetilde Y^{v}(t)=\underset{u\in \U}{\mathrm{ess}\inf}\, H(t,x,\widetilde Z^{v}(t),u,v) dt-\widetilde Z^{v}(t)dM(t),\quad 0\le t<T,\\
\widetilde Y^{v}(T)=\underset{u\in \U}{\mathrm{ess}\inf}\,h(x(T),P^{u,v}\circ x^{-1}(T)).
\end{array}
\right.
\end{equation} 
By uniqueness of the solutions of the BSDEs, we have
\begin{equation}\label{*-*}
\widehat Y^{u^*}(t)=\underset{v\in \V}{\mathrm{ess}\sup}\, Y^{u^*,v}(t),\quad \widetilde Y^{v^*}(t)=\underset{u\in \U}{\mathrm{ess}\inf}\, Y^{u,v^*}(t),
\end{equation}
and, by comparison, we have
$$
 \widehat Y^u(t)\ge Y(t) \ge \underset{v\in \V}{\mathrm{ess}\sup}\, \widetilde Y^v(t).
$$
Therefore,
$$
\underset{u\in \U}{\mathrm{ess}\inf}\, \widehat Y^u(t)\ge Y(t)\ge \underset{v\in \V}{\mathrm{ess}\sup}\, \widehat Y^v(t).
$$
But, by \eqref{sp-H} and \eqref{sp-g},  in view of the uniqueness of the solutions of the BSDEs we have
$\widehat Y^{\widehat{u}}(t)=Y(t)=\widetilde Y^{\widehat{v}}(t)$. \\ Therefore, 
$$
Y^{\widehat{u}}(t)=\underset{u\in \U}{\mathrm{ess}\inf}\, \widehat Y^u(t)=Y(t)=\widetilde Y^{\widehat{v}}(t)=\underset{v\in \V}{\mathrm{ess}\sup}\,\widetilde Y^v(t)=Y^{\widehat{u},\widehat{v}}(t).
$$
Using \eqref{*-*}, we obtain
 $$
Y(t)=Y^{\widehat{u},\widehat{v}}(t)=\underset{v\in\V}{\mathrm{ess}\sup}\,Y^{\widehat{u}, v}(t)=\underset{u\in\U}{\mathrm{ess}\inf}\,Y^{u, \widehat{v}}(t).
$$
Therefore,
$$
Y^{u, \widehat{v}}(t)\le Y^{\widehat{u}, \widehat{v}}(t)\le Y^{\widehat{u}, v}(t).
$$
Thus, $Y^{\widehat{u},\widehat{v}}(t)$ is the value of the game and $(\widehat{u},\widehat{v})$ is a saddle-point.
\end{proof}

\begin{remark}
As mentioned in Remark \eqref{selection},  If the marginal law $P^{u,v}\circ x_s^{-1}$ of $x_s$ under $P^{u,v}$ is a function of $(u(s),v(s))$ only and does not depend on the whole path of $(u,v)$ over $[0,s]$, it suffices to take the minimum and the maximum resp. of  $H$ and $h$  over the compact set  $U$ and $V$, resp.,  instead of taking the essential infimum over $\U$ and the essential maximum over $\V$. By the measurable selection theorem (see e.g. \cite{Benes}), a saddle-point  over $[0,T]$ can be obtained by pasting the saddle-points of $H$ and $h$. 

It is possible to characterize the optimal controls $\hat u$ and the equilibrium points $(\hat u,\hat v)$ in terms of a stochastic maximum principle. This approach will be discussed in a future work.
\end{remark}

%%%%%%%%%%%%APPENDIX%%%%%%%%%%%%%%%%%%
\section{Appendix}

\subsection*{Proof of \eqref{u-opt-1} in Proposition \eqref{ess-inf}}

For $n\ge 0$ let $z_n\in \mathbb{Q}^{I\times I}$, the $I\times I$-matrix with rational entries. Then, since $(t,\omega)\mapsto H(t,\omega,z_n,u)$ is $\mathbb{L}$-measurable, its essential infimum w.r.t. $u\in\U$ is well defined i.e. there exists an $\mathbb{L}$-measurable r.v. $H^n$ such that
\begin{equation}\label{H-n}
H^n(t,x, z_n)=\underset{u\in \U}{\mathrm{ess}\inf\,} H(t,x,z_n,u),\quad dP \times dt\mbox{-a.s.}
\end{equation}
Moreover, there exists a set $\mathcal{J}_n$ of $\U$ such that $(t,\omega)\mapsto \underset{u\in \mathcal{J}_n}{\inf\,} H(t,\omega,z_n,u)$ is $\mathbb{L}$-measurable and 
\begin{equation}\label{J-n-discrete}
H^n(t,x,z_n)=\underset{u\in \mathcal{J}_n}{\inf\,} H(t,x,z_n,u), \quad dP \times dt\mbox{-a.s.}
\end{equation}
Next,  set $N=\bigcup_{n\ge 0} N_n$, where
$$
N_n:=\{(t,\omega):\,\,H^n(t,\omega)\neq\underset{u\in \mathcal{J}_n}{\inf\,} H(t,\omega,z_n,u)\}.
$$
Then, $dP\otimes dt(N)=0$. \\ 

\noindent We define $H^*$ as follows: For $(t,\omega)\in N^c$  (the complement of $N$),
\begin{equation}\label{ess-inf-J}
H^*(t,x,z)=\left\{\begin{array}{ll}
\underset{u\in \mathcal{J}_n}{\inf\,} H(t,x,z_n,u) & \text{if }\,\, z=z_n\in\mathbb{Q}^{I\times I}, \\ \underset{z_n\to z}{\lim\,\,} \inf_{u\in \mathcal{J}_n} H(t,x,z_n,u) & \text{otherwise }.
\end{array}
\right.
\end{equation}
The last limit exists due to the fact that, for $n\neq m$, we have
$$
\begin{array}{ll}
|\underset{u\in \mathcal{J}_n}{\inf\,} H(t,x,z_n,u) -\underset{u\in \mathcal{J}_m}{\inf\,} H(t,x,z_m,u)|
=|H^*(t,x,z_n)-H^*(t,x,z_m)| \\ \quad \le \underset{u\in \U}{\mathrm{ess}\sup\,}\left|H(t,x,P^u\circ x^{-1}_t,z_n,u)-H(t,x,P^u\circ x^{-1}_t,z_m^,u)\right |\\ \qquad\qquad\le C(1+|x|_t+ \sup_{u\in U}\|P^u\|_2)\|z_n-z_m\|_g(t).
\end{array}
$$

\noindent  We now  show that, for every $z\in\R^{I\times I}$,
\begin{equation}\label{H-a}
H^*(t,x,z)=\underset{u\in \U}{\mathrm{ess}\inf\,} H(t,x,z,u),\quad dP \times dt\mbox{-a.s.}
\end{equation}
If $z\in\mathbb{Q}^{I\times I}$, the equality follows from the definitions \eqref{H-n} and \eqref{ess-inf-J}. Assume $z\notin\mathbb{Q}^{I\times I}$ and let $z_n\in\mathbb{Q}^{I\times I}$ such that $z_n\to z$. Further, let $\varphi(t,x_.)$ be a progressively measurable process such that $\varphi(t,x_.)\le H(t,x_.,z,u)$ for all $u\in\U$. Thus, for every $\e>0$ there exists $n_0\ge 0$ such 
$$
\begin{array}{ll}
\varphi(t,x)\le H(t,x,z_n,u)+\e,\quad n\ge n_0, \,\, u\in\U.
\end{array}
$$
Therefore, $\varphi(t,x_.)\le H^*(t,x_.,z_n)+\e,\,\,n\ge n_0$. Letting $n\to \infty$, we obtain
$\varphi(t,x_.)\le H^*(t,x_.,z)+\e$. Sending $\e$ to $0$, we finally get $\varphi(t,x_.)\le H^*(t,x_.,z)$, i.e.
$$
\underset{u\in \U}{\mathrm{ess}\inf\,} H(t,x,z,u)\le H^*(t,x,z),\quad dP \times dt\mbox{-a.s.}
$$
On the other hand, in view of \eqref{ess-inf-J} and the linearity of $H$ in $z$, we have $H^*(t,x,z)\le H(t,x,z,u),\,\, u\in\U$. Thus,
\begin{equation*}
H^*(t,x,z)\le \underset{u\in \U}{\mathrm{ess}\inf\,}H(t,x,z,u).
\end{equation*}
This finishes the proof of \eqref{H-a}. \qed

\subsection*{ Proof of \eqref{u-opt-2} in Proposition \eqref{ess-inf} }

Noting that for any $z \in \R^{I \times I}$ and $u\in \U$,  $H(t,x,z) \leq H(t,x,z,u)$, we have 
$$
H(t,x,\theta_t)\leq H(t,x,\theta_t,u), \,\quad u\in \U.
$$
Next, let $\Phi$ be an $\mathbb{L}$-measurable process such that
$\Phi(t,\omega)\leq H(t,x,\theta_t,u)$ for any $u\in \U$. Assume first that $\theta$ is uniformly bounded. Then there exists a sequence of $\mathbb{L}$-processes $(\theta^n)_{n\geq 0}$
such that for any $n\geq 0$, $\theta^n$ takes its values in $\mathbb{Q}^{I\times I}$, is piecewise constant and satisfies
 $\|\theta^n-\theta\|_\infty
:=\sup_{(t,\omega)}\|\theta^n_t(\omega)-\theta_t(\omega)
\|_g\rightarrow 0$ as $n\to\infty$.  
%(for e.g. if $d=1$, one can take
%$\theta_n=\sum_{i=1}^{n2^n}\frac{i-1}{2^n}1_{\{\theta^{-1}([\frac{i-1}{2^n},\frac{i}{2^n}[)\}}+n
%1_{\{\theta \geq n\}}$ and the generalization to the case when $d\geq 2$ is straightforward).
Furthermore,  in view of the conditions (B4) and (B8), we have
$$
|H(t,x,\theta_t,u)-H(t,x,\theta^n_t,u)|\leq C(1+\|x\|_t+\|P^u\|_2)\|\theta^n-
\theta\|_\infty.
$$
Now,  let $\epsilon >0$ and $n_0$ such that for any $n\geq n_0$,
$\|\theta^n-\theta\|_\infty\leq \epsilon$. Then, for $n\geq n_0$ and $u\in \U$ we have
$$
\Phi(t,\omega)\leq H(t,x,\theta_t^n,u)+\epsilon C(1+\|x\|_t+\|P^u\|_2),
$$
which implies that
$$
1_{B^k_n}\Phi(t,\omega)\leq
1_{B^k_n}\{H(t,x,z^k_n,u)+
\epsilon C(1+\|x\|_t+\|P^u\|_2)\},
$$
 where $B^k_n$ is a subset of $[0,T]\times \Omega$ in which $\theta_n$ is constant and equals to
$z^k_n\in \mathbb{Q}^{I\times I}$. Therefore
$$
\begin{array}{lll}
1_{B^k_n}\Phi(t,\omega)&\leq
1_{B^k_n}\{\inf_{u\in \mathcal{J}_n^k}H(t,x,z^k_n,u)+
\epsilon C(1+\|x\|_t+\|P^u\|_2)\}\\
{}& \leq 1_{B^k_n}\{\underset{u\in\U}{\mathrm{ess}\inf\,}H(t,x_.,z^k_n,u)+
\epsilon C(1+\|x\|_t+\|P^u\|_2)\}\\
{}& \leq 1_{B^k_n}\{H^*(t,x_.,z^k_n)+
\epsilon C(1+\|x\|_t+\|P^u\|_2)\}
\\
{}& \leq 1_{B^k_n}\{H^*(t,x_.,\theta^n_t)+
\epsilon C(1+\|x\|_t+\|P^u\|_2)\},
\end{array}
$$
where $\mathcal{J}_n^k$ is the countable subset of $\U$ defined in (\ref{J-n-discrete}
}) and
associated with $z^k_n$.  Summing over $k$, we obtain
\begin{equation}\label{limitssinf}\begin{array}{l}\Phi(t,\omega)\leq
H^*(t,x,\theta_t)+
2\epsilon C(1+\|x\|_t+\|P^u\|_2),
\end{array}\end{equation}
since $H^*$ is stochastic Lipschitz w.r.t. $z$ (see (\ref{H*-lip})).  Thus,
$$
|H^*(t,x,\theta_t)-H^*(t,x,\theta^n_t)|\leq \epsilon C(1+\|x\|_t +\sup_{u\in U}\|P^u\|_2)
$$ 
for $n\geq n_0$. Send now $\epsilon$ to 0 in (\ref{limitssinf}) to obtain that
$\Phi(t,\omega)\leq H^*(t,x,\theta_t)$ which means
$$
H(t,x,\theta_t)=\underset{u\in\U}{\mathrm{ess}\inf\,}H(t,x,\theta_t,u),
\,\,dP \times
dt \mbox{-a.e.}
$$
If $\theta$ is not bounded, we can find a sequence of bounded $\mathbb{L}$-processes
$(\bar \theta_n)_{n\ge 0}$ such that $\bar \theta_n\to \theta$ as $n\to \infty$, $\,\,dP \times
dt \mbox{-a.e.}$

Therefore, we have
\begin{equation}\label{reshthetan}
H^*(t,x,\bar \theta_n(t))=\underset{u\in\U}{\mathrm{ess}\inf\,}H(t,x,\bar \theta_n(t),u),
\,\,dP \times
dt \mbox{-a.e.}
\end{equation}
But, the stochastic Lipschitz property of $H^*$ and the linearity of $H$ w.r.t. $z$ imply that, as $n\to \infty$,
$$
H^*(t,x_.,\theta_n(t))\to  H^*(t,x_.,\theta(t)),\,\,
\underset{u\in\U}{\mathrm{ess}\inf\,}H(t,x,\bar \theta_n(t),u)\to
\underset{u\in\U}{\mathrm{ess}\inf\,}H(t,x,\bar \theta,u).
$$
We then obtain the desired result by taking the limit in (\ref{reshthetan}). \qed

\end{document}